\documentclass[12pt,a4paper]{amsart}

\usepackage{amsmath,amsfonts,amssymb,mathtools,extarrows}
\usepackage{xcolor}

\usepackage{color}
\usepackage{graphicx}

\usepackage{enumerate}  

\usepackage{cancel}

\usepackage[hmargin=2cm,vmargin=2cm]{geometry}

\usepackage[hypertexnames=false,hyperfootnotes=false,colorlinks=true,linkcolor=blue,%
citecolor=purple,filecolor=magenta,urlcolor=cyan,unicode,linktocpage=true,pagebackref=false]{hyperref}
\usepackage{nameref,zref-xr}     

\newcommand{\mbP}{\mathbb P}
\newcommand{\mbZ}{\mathbb Z}
\newcommand{\mbC}{\mathbb C}

\newcommand{\oM}{\overline{\mathcal M}}

\def\cM{{\mathcal{M}}}
\def\oM{{\overline{\mathcal{M}}}}

\def\mbQ{{\mathbb Q}}
\def\d{{\partial}}

\newcommand{\eps}{\varepsilon}

\newcommand{\str}{\mathrm{str}}

\newcommand{\cA}{\mathcal A}
\newcommand{\hcA}{\widehat{\mathcal A}}
\newcommand{\DR}{\mathrm{DR}}

\newcommand{\even}{\mathrm{even}}
\newcommand{\ct}{\mathrm{ct}}
\newcommand{\cF}{\mathcal F}

\newcommand{\Coef}{\mathrm{Coef}}

\renewcommand{\top}{\mathrm{top}}
\newcommand{\red}{\mathrm{red}}

\newcommand{\st}{\mathbf{H}}

\DeclareMathOperator{\Aut}{Aut}

\newcommand{\tOmega}{\widetilde\Omega}

\newcommand{\gl}{\mathrm{gl}}

\newcommand{\un}{{1\!\! 1}}
\newcommand{\mcF}{\mathcal{F}}

\DeclareMathOperator{\res}{{res}}

\DeclareMathOperator{\tdeg}{\widetilde{\mathrm{deg}}}

\newcommand{\tP}{\widetilde{P}}

\newcommand{\tw}{\widetilde{w}}

\newcommand{\tA}{\widetilde{A}}

\newcommand{\cL}{\mathcal{L}}
\newcommand{\vir}{\mathrm{vir}}
\renewcommand{\st}{\mathrm{st}}
\renewcommand{\P}{\mathrm{P}}
\newcommand{\wk}{\mathrm{wk}}
\renewcommand{\sp}{\mathrm{sp}}
\newcommand{\ozeta}{\overline{\zeta}}

\newtheorem{theorem}{Theorem}[section]
\newtheorem{proposition}[theorem]{Proposition}
\newtheorem{lemma}[theorem]{Lemma}

\theoremstyle{remark}

\newtheorem{remark}[theorem]{Remark}

\theoremstyle{definition}

\newtheorem{definition}[theorem]{Definition}

\usepackage{color}

\usepackage{comment}

\numberwithin{equation}{section}

\begin{document}

\title[Pixton's class and the noncommutative KdV hierarchy]{Intersection numbers with Pixton's class and the noncommutative KdV hierarchy}

\author{Alexandr Buryak}
\address{A. Buryak:\newline 
Faculty of Mathematics, National Research University Higher School of Economics, \newline
6 Usacheva str., Moscow, 119048, Russian Federation;\smallskip\newline 
Center for Advanced Studies, Skolkovo Institute of Science and Technology, \newline
1 Nobel str., Moscow, 143026, Russian Federation}
\email{aburyak@hse.ru}

\author{Paolo Rossi}
\address{P. Rossi:\newline Dipartimento di Matematica ``Tullio Levi-Civita'', Universit\`a degli Studi di Padova,\newline
Via Trieste 63, 35121 Padova, Italy}
\email{paolo.rossi@math.unipd.it}

\begin{abstract}
The Pixton class is a nonhomogeneous cohomology class on the moduli space of stable curves~$\oM_{g,n}$, with nontrivial terms in degree $0,2,4,\ldots,2g$, whose top degree part coincides with the double ramification cycle. In this paper, we prove our conjecture from a previous work, claiming that the generating series of intersection numbers of the Pixton class with monomials in the psi-classes gives a solution of the noncommutative KdV hierarchy.
\end{abstract}

\date{\today}

\maketitle

\section{Introduction}

Consider the moduli space of stable maps of connected algebraic curves of genus $g$ to rubber~$\mbP^1$, relative to $0,\infty\in\mbP^1$, with assigned ramification profile over these two points. Forgetting the stable map and stabilizing the source curve provides a natural morphism from this space to the moduli space of stable curves. The pushforward of the virtual fundamental class from the space of rubber maps to the space of stable curves is a degree $2g$ cohomology class called the double ramification (DR) cycle.\\

In the early 2000's Yakov Eliashberg started popularizing the problem of finding an explicit formula for the DR cycle as a tautological class in the moduli space of curves. A first partial result was Hain's formula \cite{Hai13}, which holds on the partial compactification of the moduli space of smooth curves given by curves of compact type. Analyzing the structure of Hain's formula it is not hard to see that it is reminiscent of Givental's $R$-matrix action on cohomological field theories (see \cite{PPZ15}). Building on this observation and using a clever regularization trick to cure divergencies emerging when considering curves of non compact type, Aaron Pixton was able to modify Hain's formula to find an answer to Eliashberg's problem for the full moduli space of stable maps.\\

More precisely Pixton's formula gives a nonhomogeneous tautological cohomology class on the moduli space $\oM_{g,n}$ of stable curves of genus $g$ with $n$ marked points, with terms in all even degrees. By a result from~\cite{CJ18}, the terms in the degrees greater than $2g$ vanish. In \cite{JPPZ17} it is proved that the degree~$2g$ part of this class coincides with the double ramification cycle. As Pixton points out in his 2022 ICM sectional lecture, no geometric interpretation is known for the lower degree terms of the Pixton class yet.\\

In this paper we show that the full Pixton class, including its terms of degree smaller than~$2g$, plays a beautiful role in the relation between intersection theory of the moduli space of curves and integrable systems.\\

Namely, given a ramification profile $A=(a_1,\ldots,a_n) \in \mbZ^n$, denote by $P^j_g(A) \in H^{2j}(\oM_{g,n})$ the degree $2j$ term of Pixton's class. Consider then the generating series for its intersection numbers with monomials in the psi-classes
\begin{gather*}
\cF^\P(t^*_*,\eps,\mu):=\sum_{\substack{g,n\ge 0\\2g-2+n>0}}\sum_{j=0}^g\frac{\eps^{2g}\mu^{2j}}{n!}\sum_{\substack{A=(a_1,\ldots,a_n)\in\mbZ^n\\\sum a_i=0}}\sum_{d_1,\ldots,d_n\ge 0}\left(\int_{\oM_{g,n}}2^{-j}P_g^j(A)\prod_{i=1}^n\psi_i^{d_i}\right)\prod_{i=1}^n t^{a_i}_{d_i},
\end{gather*}
and let
\begin{align*}
&(w^\P)^a:=\frac{\d^2\cF^\P}{\d t^0_0\d t^{-a}_0},\quad a\in\mbZ, \qquad w^\P:=\sum_{a\in\mbZ}(w^\P)^a e^{iay}, \qquad u^\P:=\frac{S(\eps\mu\d_x)}{S(i\eps\mu\d_x\d_y)}w^\P,
\end{align*}
where $S(z):=\frac{e^{z/2}-e^{-z/2}}{z}$. Then our main result states that $u=u^\P$ satisfies the infinite system of compatible PDEs
\begin{align}
\frac{\d u}{\d t_1}=&\d_x\left(\frac{u*u}{2}+\frac{\eps^2}{12}u_{xx}\right),\label{eq:ncKdV equation}\\
\frac{\d u}{\d t_2}=&\d_x\left(\frac{u*u*u}{6}+\frac{\eps^2}{24}(u*u_{xx}+u_x*u_x+u_{xx}*u)+\frac{\eps^4}{240}u_{xxxx}\right),\notag\\
\vdots\hspace{0.08cm}&\notag
\end{align}
where $x=t^0_0$ and $t_d = t^0_d$, which is a noncommutative analogue of the celebrated KdV hierarchy with respect to the noncommutative Moyal product
$$
f*g:=f\ \exp\left(\frac{i \eps\mu}{2}(\overleftarrow{\d_x} \overrightarrow{\d_y}-\overleftarrow{\d_y} \overrightarrow{\d_x})\right)\ g
$$
for functions $f,g$ on a $2$-dimensional torus with coordinates $x,y$ \cite{BR21a}.\\

Notice that, when $\mu=0$, the noncommutative KdV hierarchy collapses to the usual KdV hierarchy and our result recovers the celebrated Witten-Kontsevich theorem \cite{Wit91,Kon92} for integrals on the moduli space of stable curves of monomials in the psi-classes.\\

We conjectured this result in \cite{BR22}, where we also proved that it would be a consequence of the DR/DZ equivalence conjecture of \cite{BDGR18}. The latter is still open in its full generality, which states that two different constructions of integrable systems starting from a given (partial) cohomological field theory, the Dubrovin-Zhang hierarchy \cite{DZ01} and the DR hierarchy \cite{Bur15,BDGR18}, are related by a coordinate change in their phase space. However a partial result in that direction, obtained in \cite{BS22}, together with explicit computations of the flows of these two hierarchies for the partial cohomological field theory at hand, turn out to be sufficient to prove the DR/DZ equivalence specifically for the Pixton class, and hence our result. The precise structure of the proof is explained at the end of Section \ref{section:Pixton and ncKdV}.\\

Notice lastly that, together with the string and dilaton equation, equation~\eqref{eq:ncKdV equation} uniquely determines the generating series $\cF^\P$. An example was demonstrated in~\cite{BR22}, where, assuming that equation~\eqref{eq:ncKdV equation} is true, the authors proved that
\begin{gather}\label{eq:corollary of Pixton}
\sum_{1\le j\le g}\left(\int_{\oM_{g,2}}2^{-j}P_g^j(a,-a)\psi_1^{3g-1-j}\right)\mu^{2j}z^{3g-1-j}=\frac{1}{z}\left(\frac{S(a\mu z)}{S(\mu z)}e^{\frac{z^3}{24}}-1\right).
\end{gather}
Since we prove equation~\eqref{eq:ncKdV equation} in this paper, formula~\eqref{eq:corollary of Pixton} is now established.

\medskip

\subsection*{Notations and conventions}  

\begin{itemize}

\item We use the standard convention of sum over repeated Greek indices.

\smallskip

\item When it doesn't lead to a confusion, we use the symbol $*$ to indicate any value, in the appropriate range, of a sub- or superscript.

\smallskip

\item For a nonnegative integer $n$, let $[n]:=\{1,\dots,n\}$.

\smallskip

\item For a topological space $X$, we denote by $H^i(X)$ the cohomology groups with the coefficients in $\mbC$. Let $H^{\even}(X):=\bigoplus_{i\ge 0}H^{2i}(X)$.

\smallskip

\item We will work with the moduli spaces $\oM_{g,n}$ of stable algebraic curves of genus $g$ with $n$ marked points, which are defined for $g,n\ge 0$ satisfying the condition $2g-2+n>0$. We will often omit mentioning this condition explicitly, and silently assume that it is satisfied when a moduli space is considered. 

\end{itemize}

\medskip

\subsection*{Acknowledgements}

The work of A.~B. is an output of a research project implemented as part of the Basic Research Program at the National Research University Higher School of Economics (HSE University). P.~R. was partially supported by the BIRD-SID-2021 UniPD grant and is affiliated to GNSAGA--INDAM (Istituto Nazionale di Alta Matematica) and to INFN (Istituto Nazionale di Fisica Nucleare).

\medskip

%%%%%%%%%%%%%%%%%%%%%%%%%%%%%%%%%%%%%%%%%%%%%%%%%%%%%%%%%%%%%%%%%%%%
%%%%%%%%%%%%%%%%%%%%%%%%%%%%%%%%%%%%%%%%%%%%%%%%%%%%%%%%%%%%%%%%%%%%

\section{The Pixton class and the noncommutative KdV hierarchy}\label{section:Pixton and ncKdV}

\subsection{The double ramification cycle and the Pixton class}

Denote by $\psi_i\in H^2(\oM_{g,n})$ the first Chern class of the line bundle $\cL_i$ over~$\oM_{g,n}$ formed by the cotangent lines at the $i$-th marked point on stable curves. The classes $\psi_i$ are called the {\it psi-classes}. Denote by~$\mathbb E$ the rank~$g$ {\it Hodge vector bundle} over~$\oM_{g,n}$ whose fibers are the spaces of holomorphic one-forms on stable curves. Let $\lambda_j:=c_j(\mathbb E)\in H^{2j}(\oM_{g,n})$. 

\medskip

Consider an $n$-tuple of integers $A=(a_1,\ldots,a_n)$ such that $\sum a_i=0$; it will be called a {\it vector of double ramification data}. The positive parts of $A$ define a partition $\mu=(\mu_1,\ldots,\mu_{l(\mu)})$. The negative parts of $A$ define a second partition $\nu=(\nu_1,\ldots,\nu_{l(\nu)})$. Since the parts of $A$ sum to~$0$, the partitions $\mu$ and $\nu$ must be of the same size. We allow the case $|\mu|=|\nu|=0$. Let $n_0:=n-l(\mu)-l(\nu)$. The moduli space 
$$
\oM_{g,n_0}(\mbP^1,\mu,\nu)^\sim
$$
parameterizes stable relative maps of connected algebraic curves of genus $g$ to rubber $\mbP^1$ with ramification profiles $\mu,\nu$ over the points $0,\infty\in\mbP^1$, respectively. There is a natural map 
$$
\st\colon \oM_{g,n_0}(\mbP^1,\mu,\nu)^\sim\to\oM_{g,n}
$$
forgetting everything except the marked domain curve. The moduli space $\oM_{g,n_0}(\mbP^1,\mu,\nu)^\sim$ possesses a virtual fundamental class $\left[\oM_{g,n_0}(\mbP^1,\mu,\nu)^\sim\right]^\vir$, which is a homology class of degree $2(2g-3+n)$. The {\it double ramification cycle} 
$$
\DR_g(A)\in H^{2g}(\oM_{g,n})
$$
is defined as the Poincar\'e dual to the push-forward $\st_*\left[\oM_{g,n_0}(\mbP^1,\mu,\nu)^\sim\right]^\vir\in H_{2(2g-3+n)}(\oM_{g,n})$.

\medskip

Let us now recall Pixton's very explicit construction of a nonhomogeneous cohomology class on $\oM_{g,n}$, with nontrivial terms in degree $0,2,4,\ldots,2g$. By a result of~\cite{JPPZ17}, the degree~$2g$ part of this class coincides with the double ramification cycle. 

\medskip

Let us first recall a standard way to construct cohomology classes on $\oM_{g,n}$ in terms of stable graphs. A {\it stable graph} is the following data:
$$
\Gamma=(V,H,L,g\colon V\to\mbZ_{\ge 0},v\colon H\to V,\iota\colon H\to H),
$$
where 
\begin{enumerate}
\item $V$ is a set of {\it vertices} with a genus function $g\colon V\to\mbZ_{\ge 0}$;

\smallskip

\item $H$ is a set of {\it half-edges} equipped with a vertex assignment $v\colon H\to V$ and an involution~$\iota$;

\smallskip

\item the set of {\it edges} $E$ is defined as the set of orbits of $\iota$ of length $2$;

\smallskip

\item the set of {\it legs} $L$ is defined as the set of fixed points of $\iota$ and is placed in bijective correspondence with the set $[n]$, the leg corresponding to the marking $i\in[n]$ will be denoted by $l_i$;

\smallskip

\item the pair $(V,E)$ defines a connected graph;

\smallskip

\item the stability condition $2g(v)-2+n(v)>0$ is satisfied at each vertex $v\in V$, where $n(v)$ is the valence of $\Gamma$ at $v$ including both half-edges and legs.  
\end{enumerate}
An {\it automorphism} of $\Gamma$ consists of automorphisms of the sets $V$ and $H$ that leave invariant the structures $L,g,v$, and $\iota$. Denote by $\Aut(\Gamma)$ the authomorphism group of $\Gamma$. The {\it genus} of a stable graph $\Gamma$ is defined by $g(\Gamma):=\sum_{v\in V}g(v)+h^1(\Gamma)$. Denote by $G_{g,n}$ the set of isomorphism classes of stable graphs of genus $g$ with $n$ legs.

\medskip

For each stable graph $\Gamma\in G_{g,n}$, there is an associated moduli space
$$
\oM_\Gamma:=\prod_{v\in V}\oM_{g(v),n(v)}
$$
and a canonical map
$$
\xi_\Gamma\colon \oM_\Gamma\to\oM_{g,n}
$$
that is given by the gluing of the marked points corresponding to the two halves of each edge in $E(\Gamma)$. Each half-edge $h\in H(\Gamma)$ determines a cotangent line bundle $\cL_h\to\oM_\Gamma$. If $h\in L(\Gamma)$, then $\cL_h$ is the pull-back via $\xi_\Gamma$ of the corresponding cotangent line bundle over~$\oM_{g,n}$. Let $\psi_h:=c_1(\cL_h)\in H^2(\oM_\Gamma)$. The Pixton class will be described as a linear combination of cohomology classes of the form
$$
\xi_{\Gamma*}\left(\prod_{h\in H}\psi_h^{d(h)}\right),
$$
where $\Gamma\in G_{g,n}$ and $d\colon H(\Gamma)\to\mbZ_{\ge 0}$.

\medskip 

Let $A=(a_1,\ldots,a_n)$ be a vector of double ramification data. Let $\Gamma\in G_{g,n}$ and $r\ge 1$. A {\it weighting mod $r$} of $\Gamma$ is a function
$$
w\colon H(\Gamma)\to \{0,\ldots,r-1\}
$$
that satisfies the following three properties:
\begin{enumerate}
\item for any leg $l_i\in L(\Gamma)$, we have $w(l_i)=a_i\mod r$;

\smallskip

\item for any edge $e=\{h,h'\}\in E(\Gamma)$, we have $w(h)+w(h')=0\mod r$;

\smallskip

\item for any vertex $v\in V(\Gamma)$, we have $\sum_{h\in H(\Gamma),\,v(h)=v}w(h)=0\mod r$.
\end{enumerate}
Denote by $W_{\Gamma,r}$ the set of weightings mod $r$ of $\Gamma$. We have $|W_{\Gamma,r}|=r^{h^1(\Gamma)}$.

\medskip

We denote by $P_g^{d,r}(A)\in H^{2d}(\oM_{g,n})$ the degree $2d$ component of the cohomology class 
\begin{equation}\label{eq:Pixton}
\sum_{\Gamma\in G_{g,n}}\sum_{w\in W_{\Gamma,r}}\frac{1}{|\Aut(\Gamma)|}\frac{1}{r^{h^1(\Gamma)}}\xi_{\Gamma*}\left[\prod_{i=1}^n\exp(a_i^2\psi_{l_i})\prod_{e=\{h,h'\}\in E(\Gamma)}\frac{1-\exp\left(-w(h)w(h')(\psi_h+\psi_{h'})\right)}{\psi_h+\psi_{h'}}\right]
\end{equation}
in $H^*(\oM_{g,n})$. Note that the factor $\frac{1-\exp\left(-w(h)w(h')(\psi_h+\psi_{h'})\right)}{\psi_h+\psi_{h'}}$ is well defined since the denominator formally divides the numerator. In~\cite{JPPZ17}, the authors proved that for fixed $g,A$, and $d$ the class $P_g^{d,r}$ is polynomial in~$r$ for all sufficiently large~$r$. Denote by $P_g^d(A)$ the constant term of the associated polynomial in~$r$.

\medskip

In~\cite{JPPZ17}, the authors proved that 
$$
\DR_g(A)=2^{-g}P^g_g(A).
$$
In~\cite{CJ18}, the authors proved that the class $P_g^d(A)$ vanishes for $d>g$. 

\medskip

\subsection{The noncommutative KdV hierarchy}

The classical construction of the \emph{KdV hierarchy} as the system of Lax equations (see, e.g.,~\cite{Dic03})
$$
\frac{\d L}{\d t_n}=\frac{\eps^{2n}}{(2n+1)!!}\left[\left(L^{n+1/2}\right)_+,L\right],\quad n\ge 1,
$$
where $L:=\d_x^2+2\eps^{-2}u$, $u$ is a function of $x,t_1,t_2,\ldots$, $\eps$ is a formal parameter, and $(2n+1)!!:=(2n+1)\cdot(2n-1)\cdots 3\cdot 1$, admits generalizations, called \emph{noncommutative KdV hierarchies}, where one doesn't have the pairwise commutativity of the $x$-derivatives of the dependent variable~$u$. In what follows, we will work with a specific example from the class of noncommutative KdV hierarchies.

\medskip

Let $u_{k_1,k_2}$, $k_1,k_2\in\mbZ_{\ge 0}$, $\eps$, and $\mu$ be formal variables and consider the algebra $\hcA:=\mbC[[u_{*,*},\eps,\mu]]$, whose elements will be called {\it differential polynomials in two space variables}. Consider a gradation on $\hcA$ given by 
$$
\deg u_{k_1,k_2}:=(k_1,k_2),\qquad \deg\eps:=(-1,0),\qquad \deg\mu:=(0,-1).
$$
We will denote by $\hcA^{[(d_1,d_2)]}\subset\hcA$ the space of differential polynomials of degree $(d_1,d_2)$. The space $\hcA$ is endowed with operators $\d_x$ and $\d_y$ of degrees $(1,0)$ and $(0,1)$, respectively, defined~by
\begin{gather*}
\d_x:=\sum_{k_1,k_2\geq 0} u_{k_1+1,k_2} \frac{\d}{\d u_{k_1,k_2}},\qquad \d_y:=\sum_{k_1,k_2\geq 0} u_{k_1,k_2+1} \frac{\d}{\d u_{k_1,k_2}}.
\end{gather*}
We see that $u_{k_1,k_2}=\d_x^{k_1}\d_y^{k_2}u_{0,0}$. We will denote $u_{0,0}$ simply by $u$.

\medskip

The algebra $\hcA$ is also endowed with the {\it Moyal star-product} defined by
\begin{gather}\label{eq:Moyal}
f*g:=f\ \exp\left(\frac{i \eps\mu}{2}(\overleftarrow{\d_x} \overrightarrow{\d_y}-\overleftarrow{\d_y} \overrightarrow{\d_x})\right)\ g =\sum_{k_1,k_2\ge 0}\frac{(-1)^{k_2}(i\eps\mu)^{k_1+k_2}}{2^{k_1+k_2} k_1!k_2!}(\d_x^{k_1} \d_y^{k_2} f) (\d_x^{k_2}\d_y^{k_1} g),
\end{gather}
where $f,g \in \mbC[[u_{*,*},\eps,\mu]]$. The Moyal star-product is associative and it is graded: if $\deg f = (i_1,i_2)$ and $\deg g = (j_1,j_2)$, then $\deg{(f * g)} = (i_1+j_1,i_2+j_2)$. Note also that when $\mu=0$ the Moyal star-product becomes the usual multiplication: 
\begin{gather}\label{eq:Moyal at mu=0}
\left.(f*g)\right|_{\mu=0}=f|_{\mu=0}\cdot g|_{\mu=0}.
\end{gather}

\medskip

Let us now consider the algebra of \emph{pseudo-differential operators} of the form
\begin{gather}\label{eq:pseudo-differential operator}
A=\sum_{i\le n}a_i*\d_x^i,\quad n\in\mbZ,\quad a_i\in\mbC[[u_{*,*},\mu]][[\eps,\eps^{-1}],
\end{gather}
with the multiplication $\circ$ given by
$$
(a*\d_x^i)\circ(b*\d_x^j):=\sum_{k\ge 0}{i\choose k}(a*\d_x^k b)*\d_x^{i+j-k},\quad a,b\in \mbC[[u_{*,*},\mu]][[\eps,\eps^{-1}],\quad i,j\in\mbZ.
$$
The positive part of a pseudo-differential operator~\eqref{eq:pseudo-differential operator} is defined by $A_+:=\sum_{0\le i\le n}a_i*\d_x^i$ and, as in the classical theory of pseudo-differential operators, a pseudo-differential operator $A$ of the form $\d_x^2+\sum_{i<2}a_i*\d_x^i$ has a unique square root of the form $\d_x+\sum_{i<1}b_i*\d_x^i$, which we denote by $A^{\frac{1}{2}}$.

\medskip

Consider the operator $L:=\d_x^2+2\eps^{-2}u$. The {\it noncommutative KdV hierarchy} with respect to the Moyal star-product~\eqref{eq:Moyal} is defined by (see, e.g.,~\cite{Ham05,DM00})
\begin{gather}\label{eq:ncKdV hierarchy}
\frac{\d L}{\d t_n}=\frac{\eps^{2n}}{(2n+1)!!}\left[\left(L^{n+1/2}\right)_+,L\right],\quad n\ge 1.
\end{gather}
The noncommutative KdV hierarchy is integrable in the sense that its flows pairwise commute. Explicitly, the first two equations of the hierarchy are
\begin{align*}
\frac{\d u}{\d t_1}=&\d_x\left(\frac{u*u}{2}+\frac{\eps^2}{12}u_{xx}\right),\\
\frac{\d u}{\d t_2}=&\d_x\left(\frac{u*u*u}{6}+\frac{\eps^2}{24}(u*u_{xx}+u_x*u_x+u_{xx}*u)+\frac{\eps^4}{240}u_{xxxx}\right).
\end{align*}

\medskip

\begin{remark}\label{remark:about ncKdV}
We see that the noncommutative KdV hierarchy is a system of evolutionary PDEs with one dependent variable~$u$ and two spatial variables $x$ and $y$. Note that after the substitution $u=\sum_{a\in\mbZ}u^a e^{iay}$, where $u^a$ are new formal variables, the noncommutative KdV hierarchy becomes a system of evolutionary PDEs with infinitely many dependent variables~$u^a$, $a\in\mbZ$, and one spatial variable $x$. Let us present, for example, the resulting equations for the flow~$\frac{\d}{\d t_1}$:
\begin{gather}\label{eq:ncKdV,flow1,with x}
\frac{\d u^a}{\d t_1}=\d_x\left(\sum_{\substack{a_1,a_2\in\mbZ\\a_1+a_2=a}}e^{-\frac{a_2}{2}\eps\mu\d_x}u^{a_1}\cdot e^{\frac{a_1}{2}\eps\mu\d_x}u^{a_2}\right)+\frac{\eps^2}{12}u^a_{xxx},\quad a\in\mbZ.
\end{gather}
\end{remark}

\begin{comment}
For any $n\ge 1$, the right-hand side of~\eqref{eq:ncKdV hierarchy} has the form $\d_x P_n$, where $P_n\in\hcA^{[(0,0)]}$. Moreover, $P_n=\sum_{g=0}^n P_{n,g}$, where $P_{n,g}$ is a linear combination of the monomials $\eps^{2g}u_{d_1}*\cdots * u_{d_{n+1-g}}$ with $d_1+\cdots+d_{n+1-g}=2g$. The leading term $P_{n,0}$ is equal to $\frac{u^{*(n+1)}}{(n+1)!}$. The hierarchy 
\begin{gather}\label{eq:dncKdV}
\frac{\d u}{\d t_n}=\d_x\left(\frac{u^{*(n+1)}}{(n+1)!}\right),\quad n\ge 1,
\end{gather}
will be called the {\it dispersionless noncommutative KdV (dncKdV) hierarchy}.
\end{comment}

\medskip

Note that because of~\eqref{eq:Moyal at mu=0} the noncommutative KdV hierarchy becomes the classical KdV hierarchy when $\mu=0$.

\medskip

\subsection{Main result}

Consider formal variables $t^a_d$, where $a\in\mbZ$ and $d\in\mbZ_{\ge 0}$. Let us introduce the following generating series:
\begin{gather*}
\cF^\P(t^*_*,\eps,\mu):=\sum_{g,n\ge 0}\sum_{j=0}^g\frac{\eps^{2g}\mu^{2j}}{n!2^j}\sum_{\substack{A=(a_1,\ldots,a_n)\in\mbZ^n\\\sum a_i=0}}\sum_{d_1,\ldots,d_n\ge 0}\left(\int_{\oM_{g,n}}P_g^j(A)\prod_{i=1}^n\psi_i^{d_i}\right)\prod_{i=1}^n t^{a_i}_{d_i}\in\mbC[[t^*_*,\eps,\mu]].
\end{gather*}
Introduce a formal power series 
$$
S(z):=\frac{e^{z/2}-e^{-z/2}}{z}=1+\frac{z^2}{24}+\frac{z^4}{1920}+O(z^6)
$$
and let
\begin{align}
&w^{\P;a}:=\frac{\d^2\cF^\P}{\d t^0_0\d t^{-a}_0}\in\mbC[[t^*_*,\eps,\mu]],\quad a\in\mbZ,\notag\\
&w^\P:=\sum_{a\in\mbZ}w^{\P;a} e^{iay}\in\mbC[[t^*_*,\eps,\mu]][[e^{iy},e^{-iy}]],\notag\\
&u^\P:=\frac{S(\eps\mu\d_x)}{S(i\eps\mu\d_x\d_y)}w^\P\in\mbC[[t^*_*,\eps,\mu]][[e^{iy},e^{-iy}]].\label{eq:uP and wP}
\end{align}
Here we identify $x=t^0_0$, so the operator $\d_x$ acts as the partial derivative $\frac{\d}{\d t^0_0}$.

\medskip

The following theorem was conjectured in~\cite[Conjecture~2]{BR22}.

\begin{theorem}\label{theorem:main}
The series $u^\P$ satisfies the noncommutative KdV hierarchy~\eqref{eq:ncKdV hierarchy}, where we identify $t^0_d=t_d$ and $t^0_0=x$.
\end{theorem}

\begin{remark}
Introduce formal power series $T_a(z)$, $a\in\mbZ$, by $T_a(z):=\frac{S(z)}{S(a z)}$, and consider the expansion $u^{\P}=\sum_{a\in\mbZ}u^{\P;a}e^{iay}$. Then formula~\eqref{eq:uP and wP} implies that
$$
u^{\P;a}=T_a(\eps\mu\d_x)w^{\P;a}.
$$
So, equivalently, Theorem~\ref{theorem:main} says that the collection of formal powers series $u^{\P;a}$, $a\in\mbZ$, satisfies the noncommutative KdV hierarchy viewed as a system of evolutionary PDEs with one spatial variable, see Remark~\ref{remark:about ncKdV}. 
\end{remark}

\medskip

The proof of this theorem is organised as follows. 

\medskip

To any finite rank partial cohomological field theory (CohFT), one can associate two systems of evolutionary PDEs with one spatial variable: the Dubrovin--Zhang (DZ) hierarchy~\cite{DZ01,BS22} and the double ramification (DR) hierarchy \cite{Bur15,BDGR18,BS22}. We will recall the necessary definitions and constructions in Sections~\ref{subsection:DZ hierarchy} and~\ref{subsection:DR hierarchy}. Conjecturally, the two hierarchies are Miura equivalent. This is not proved yet, but there is a partial result in this direction, obtained in~\cite{BS22}, which we will recall in Section~\ref{subsection:relation between two hierarchies}. After that, in Section~\ref{subsection:reconstruction from one flow}, we will describe a class of integrable systems, containing the DZ and DR hierarchies, having the property that any integrable system from this class is uniquely determined by the equation of one particular flow, which we call the \emph{special flow}. This is a slight generalization of a result from~\cite[Section~5.2]{BG16}. Pixton's classes form an infinite rank partial CohFT~\cite[Proposition~4.6]{BR22}. Extending the constructions and results, developed for the finite rank case, to the infinite rank case requires some care, as it was already mentioned in~\cite[Section~3]{BR21a} and~\cite[Section~4.1]{BR22}. In Section~\ref{subsection:comments about infinite rank}, we will discuss how to extend the results that we need to the infinite rank case.  

\medskip

After the preparatory work in Section~\ref{section:hierarchies for pCohFTs}, the proof of Theorem~\ref{theorem:main} is presented in Section~\ref{section:proof of the main theorem}. In~\cite{BR21a}, the authors proved that a part of the DR hierarchy corresponding to the partial CohFT formed by Pixton's classes coincides with the noncommutative KdV hierarchy, viewed as a system of evolutionary PDEs with one spatial variable, see Remark~\ref{remark:about ncKdV}. Therefore, for the proof of Theorem~\ref{theorem:main}, it is sufficient to prove that the DR and DZ hierarchies are Miura equivalent. In Section~\ref{subsection:using DR hierarchy}, we will show that the Miura equivalence follows from Proposition~\ref{proposition:main}, which says that the DZ hierarchy is polynomial and explicitly describes the flow $\frac{\d}{\d t^0_1}$. The two parts of the proposition are then proved in Sections~\ref{subsection:part 1} and~\ref{subsection:part 2}, respectively, using results from~\cite{BPS12},~\cite{BSSZ15},~and~\cite{BS22}. This will complete the proof of Theorem~\ref{theorem:main}.

\medskip

%%%%%%%%%%%%%%%%%%%%%%%%%%%%%%%%%%%%%%%%%%%%%%%%%%%%%%%%%%%%%%%%%%%%
%%%%%%%%%%%%%%%%%%%%%%%%%%%%%%%%%%%%%%%%%%%%%%%%%%%%%%%%%%%%%%%%%%%%

\section{Hierarchies associated to a partial cohomological field theory}\label{section:hierarchies for pCohFTs}

\begin{definition}[\cite{LRZ15}]
A \emph{partial cohomological field theory (CohFT)} is a system of linear maps 
$$
c_{g,n}\colon V^{\otimes n} \to H^\even(\oM_{g,n}),\quad g,n\ge 0,\quad 2g-2+n>0,
$$
where $V$ is an arbitrary finite dimensional vector space, together with a special element $e\in V$, called the \emph{unit}, and a symmetric nondegenerate bilinear form $\eta\in (V^*)^{\otimes 2}$, called the \emph{metric}, such that the following axioms are satisfied:
\begin{itemize}
\item[(i)] The maps $c_{g,n}$ are equivariant with respect to the $S_n$-action permuting the $n$ copies of~$V$ in $V^{\otimes n}$ and the $n$ marked points in $\oM_{g,n}$, respectively.

\smallskip

\item[(ii)] $\pi^* c_{g,n}(\otimes_{i=1}^n v_i) = c_{g,n+1}(\otimes_{i=1}^n  v_i\otimes e)$ for $v_1,\ldots,v_n\in V$, where $\pi\colon\oM_{g,n+1}\to\oM_{g,n}$ is the map that forgets the last marked point. Moreover, $c_{0,3}(v_1\otimes v_2 \otimes e) =\eta(v_1\otimes v_2)$ for $v_1,v_2\in V$.

\smallskip

\item[(iii)] Choosing a basis $e_1,\ldots,e_{\dim V}$ of $V$, we have
$$
\gl^* c_{g_1+g_2,n_1+n_2}( \otimes_{i=1}^{n_1+n_2} e_{\alpha_i}) = \eta^{\mu \nu}c_{g_1,n_1+1}(\otimes_{i\in I} e_{\alpha_i} \otimes e_\mu)\otimes c_{g_2,n_2+1}(\otimes_{j\in J} e_{\alpha_j}\otimes e_\nu)
$$
for $1\leq\alpha_1,\ldots,\alpha_{n_1+n_2}\leq \dim V$, where $I \sqcup J =[n_1+n_2]$, $|I|=n_1$, $|J|=n_2$, and $\gl\colon\oM_{g_1,n_1+1}\times\oM_{g_2,n_2+1}\to \oM_{g_1+g_2,n_1+n_2}$ is the corresponding gluing map, and where $\eta_{\alpha\beta}:=\eta(e_\alpha\otimes e_\beta)$ and~$\eta^{\alpha\beta}$ is defined by $\eta^{\alpha \mu}\eta_{\mu \beta} = \delta^\alpha_\beta$ for $1\leq \alpha,\beta\leq \dim V$. Clearly the axiom doesn't depend on the choice of a basis in $V$.
\end{itemize}
The dimension of $V$ is called the \emph{rank} of the partial CohFT.
\end{definition}

\medskip

For any partial CohFT, the tensor $c^\alpha_{\beta\gamma}:=\eta^{\alpha\mu}c_{0,3}(e_\mu\otimes e_\beta\otimes e_\gamma)$ defines a commutative associative algebra. The partial CohFT is called \emph{semisimple} if this algebra doesn't have nilpotents.

\medskip

\begin{definition}[\cite{KM94}]
A \emph{CohFT} is a partial CohFT $c_{g,n}\colon V^{\otimes n} \to H^\even(\oM_{g,n})$ such that the following extra axiom is satisfied:
\begin{itemize}
\item[(iv)] $\gl^* c_{g+1,n}(\otimes_{i=1}^n e_{\alpha_i}) = c_{g,n+2}(\otimes_{i=1}^n e_{\alpha_i}\otimes e_{\mu}\otimes e_\nu) \eta^{\mu \nu}$ for $1 \leq\alpha_1,\ldots,\alpha_n\leq \dim V$, where  $\gl\colon\oM_{g,n+2}\to \oM_{g+1,n}$ is the gluing map that increases the genus by identifying the last two marked points.
\end{itemize}
\end{definition}

\medskip

\subsection{The Dubrovin--Zhang hierarchy}\label{subsection:DZ hierarchy}

A construction of a dispersive hierarchy, the so-called \emph{Dubrovin--Zhang (DZ) hierarchy}, associated to an arbitrary homogeneous semisimple Dubrovin--Frobenius manifold was developed in a series of papers of Dubrovin and Zhang~\cite{Dub96,DZ98,DZ99,DZ01}. Perhaps the most complete exposition of their theory is contained in~\cite{DZ01}. An equivalent approach, which can also be generalized to the nonhomogeneous case, was presented in~\cite{BPS12}, where the authors also proved the polynomiality property of the equations of the DZ hierarchy. A further generalization of the construction of the DZ hierarchy was given in~\cite{BS22}, where the authors also clarified a role of certain relations in the cohomology of $\oM_{g,n}$ for the polynomiality property of the DZ hierarchy. In this section, we present the necessary definitions and results from the theory of DZ hierarchies, following~\cite{BS22}.

\medskip

\subsubsection{Construction}

Let us fix an integer $N\ge 1$ and consider formal variables $w^1,\ldots,w^N$. Let us briefly recall main notions and notations in the formal theory of evolutionary PDEs with one spatial variable:
\begin{itemize}
\item To the formal variables $w^\alpha$ we attach formal variables $w^\alpha_d$ with $d\ge 0$ and introduce the algebra of \emph{differential polynomials} $\cA_w:=\mbC[[w^*]][w^*_{\ge 1}]$. We identify $w^\alpha_0=w^\alpha$ and also denote $w^\alpha_x:=w^\alpha_1$, $w^{\alpha}_{xx}:=w^\alpha_2$, \ldots.

\smallskip

\item An operator $\d_x\colon\cA_w\to\cA_w$ is defined by $\d_x:=\sum_{d\ge 0}w^\alpha_{d+1}\frac{\d}{\d w^\alpha_d}$.

\smallskip

\item $\cA_{w;d}\subset\cA_w$ is the homogeneous component of (differential) degree $d$, where $\deg w^\alpha_i:=i$. %For $f\in\cA_w$ we denote by $f^{[d]}\in\cA_{w;d}$ the image of $f$ under the canonical projection $\cA_w\to\cA_{w;d}$. We will also use the notation $\cA_{w;\le d}:=\bigoplus_{i=0}^d\cA_{w;i}$.

\smallskip

\item The extended space of differential polynomials is defined by $\hcA_w:= \cA_w[[\eps]]$. Let $\hcA_{w;k}\subset\hcA_w$ be the homogeneous component of degree~$k$, where $\deg\eps:=-1$.  

\smallskip

\item A system of \emph{evolutionary PDEs} (with one spatial variable) is a system of equations of the form $\frac{\d w^\alpha}{\d t}=P^\alpha$, $1\le\alpha\le N$, where $P^\alpha\in\hcA_w$. Two systems $\frac{\d w^\alpha}{\d t}=P^\alpha$ and $\frac{\d w^\alpha}{\d s}=Q^\alpha$ are said to be \emph{compatible} (or, equivalently, that the flows $\frac{\d}{\d t}$ and $\frac{\d}{\d s}$ \emph{commute}) if $\sum_{n\ge 0}\left(\frac{\d P^\alpha}{\d w^\beta_n}\d_x^n Q^\beta-\frac{\d Q^\alpha}{\d w^\beta_n}\d_x^n P^\beta\right)=0$ for any $1\le\alpha\le N$. 

\smallskip

\item A \emph{Miura transformation} (that is close to identity) is a change of variables $w^\alpha\mapsto \tw^\alpha(w^*_*,\eps)$ of the form $\tw^\alpha(w^*_*,\eps)=w^\alpha+\eps f^\alpha(w^*_*,\eps)$, where $f^\alpha\in\hcA_{w;1}$. 
\end{itemize}

\medskip

Consider an arbitrary partial CohFT $\{c_{g,n}\colon V^{\otimes n}\to H^{\even}(\oM_{g,n})\}$ with $\dim V=N$, metric $\eta\colon V^{\otimes 2}\to \mbC$, and unit $e\in V$. We fix a basis $e_1,\ldots,e_N\in V$, consider formal variables $t^\alpha_d$ with $1\le\alpha\le N$ and $d\ge 0$, and define the \emph{potential} of our partial CohFT by
\begin{gather*}
\mcF:=\sum_{g,n\ge 0}\frac{\eps^{2g}}{n!}\sum_{\substack{1\le\alpha_1,\ldots,\alpha_n\le N\\d_1,\ldots,d_n\ge 0}}\left(\int_{\oM_{g,n}}c_{g,n}(\otimes_{i=1}^n e_{\alpha_i})\prod_{i=1}^n\psi_i^{d_i}\right)\prod_{i=1}^n t^{\alpha_i}_{d_i}\in\mbC[[t^*_*,\eps]].
\end{gather*}
The potential satisfies the \emph{string} and the \emph{dilaton} equations:
\begin{align}
&\frac{\d\mcF}{\d t^\un_0}=\sum_{n\ge 0}t^\alpha_{n+1}\frac{\d\mcF}{\d t^\alpha_n}+\frac{1}{2}\eta_{\alpha\beta}t^\alpha_0 t^\beta_0+\eps^2\int_{\oM_{1,1}}c_{1,1}(e),\label{eq:string equation}\\
&\frac{\d\mcF}{\d t^\un_1}=\sum_{n\ge 0}t^\alpha_n\frac{\d\mcF}{\d t^\alpha_n}+\eps\frac{\d\mcF}{\d\eps}-2\mcF+\eps^2\int_{\oM_{1,1}}\psi_1 c_{1,1}(e),\label{eq:dilaton equation}
\end{align}
where $\frac{\d}{\d t^\un_0}:=A^\mu\frac{\d}{\d t^\mu_0}$, and the coefficients $A^\mu$ are given by $e=A^\mu e_\mu$. Let us also define formal power series $w^{\top;\alpha}:=\eta^{\alpha\mu}\frac{\d^2\mcF}{\d t^\mu_0\d t^\un_0}$, $w^{\top;\alpha}_n:=\frac{\d^n w^{\top;\alpha}}{(\d t^\un_0)^n}$.

\medskip

For $d\ge 0$, denote by $\mbC[[t^*_*]]^{(d)}$ the subset of $\mbC[[t^*_*]]$ formed by infinite linear combinations of monomials $\prod t^{\alpha_i}_{d_i}$ with $\sum d_i\ge d$. Clearly, $\mbC[[t^*_*]]^{(d)}\subset \mbC[[t^*_*]]$ is an ideal. From the string equation~\eqref{eq:string equation} it follows that
\begin{gather}\label{eq:main property of wtop}
w^{\top;\alpha}_n=t^\alpha_n+\delta_{n,1}A^\alpha+R^\alpha_n(t^*_*)+O(\eps^2)\quad \text{for some $R^\alpha_n\in\mbC[[t^*_*]]^{(n+1)}$}.
\end{gather}
This implies that any formal power series in the variables $t^\alpha_a$ and $\eps$ can be expressed as a formal power series in $\left(w^{\top;\beta}_b-\delta_{b,1}A^\beta\right)$ and $\eps$ in a unique way. 

\medskip

So we denote by $\cA^\wk_w$ the ring of formal power series in the shifted variables $(w^\alpha_n-A^\alpha\delta_{n,1})$, and let $\hcA^\wk_w:=\cA^\wk_w[[\eps]]$. We have the obvious inclusion $\hcA_w\subset\hcA^\wk_w$. We see that for any $(\alpha,a),(\beta,b)\in[N]\times\mbZ_{\ge 0}$ there exists a unique element $\Omega^{\alpha,a}_{\beta,b}\in\hcA^\wk_w$ such that
$$
\eta^{\alpha\mu}\frac{\d^2\mcF}{\d t^\mu_a\d t^{\beta}_{b}}=\left.\Omega^{\alpha,a}_{\beta,b}\right|_{w^\gamma_c=w^{\top;\gamma}_c}.
$$
Clearly we have $\frac{\d w^{\top;\alpha}}{\d t^\beta_b}=\left.\d_x\Omega^{\alpha,0}_{\beta,b}\right|_{w^\gamma_c=w^{\top;\gamma}_c}$. Note also that $\Omega^{\alpha,0}_{\un,0}=w^\alpha$. Therefore, the $N$-tuple of formal powers series $\left.w^{\top;\alpha}\right|_{t^\gamma_0\mapsto t^\gamma_0+A^\gamma x}$ satisfies the system of generalized PDEs
\begin{gather}\label{eq:DZ for F-CohFT}
\frac{\d w^\alpha}{\d t^\beta_b}=\d_x\Omega^{\alpha,0}_{\beta,b},\quad 1\le\alpha,\beta\le N,\, b\ge 0,
\end{gather}
which we call the \emph{Dubrovin--Zhang (DZ) hierarchy} associated to our partial CohFT. We say ``generalized PDEs'', because the right-hand sides are not differential polynomials, but elements of the larger algebra $\hcA^\wk_w$. The $N$-tuple $(w^{\top;1},\ldots,w^{\top;N})$ will be called the \emph{topological solution} of the DZ hierarchy.

\medskip

Conjecturally, $\Omega^{\alpha,a}_{\beta,b}$ are differential polynomials, which is proved when the partial CohFT is a semisimple CohFT~\cite[Theorem~22]{BPS12}. If, for a given partial CohFT, $\Omega^{\alpha,a}_{\beta,b}$ are differential polynomials, then the flows of the DZ hierarchy pairwise commute.

\medskip

\subsubsection{Polynomiality property and relations in the cohomology of $\oM_{g,n}$}

Denote
$$
\mcF^{\alpha,a}:=\eta^{\alpha\mu}\frac{\d\mcF}{\d t^\mu_a}\in\mbC[[t^*_*,\eps]],\quad \alpha\in[N],\,a\in\mbZ_{\ge 0}.
$$
By \cite[Theorem~4.6]{BS22}, there exists a unique differential polynomial $\tOmega^{\alpha,a}_{\beta,b}\in\hcA_{w;0}$ such that the difference
$$
\Omega^{\red;\alpha,a}_{\beta,b}:=\frac{\d\mcF^{\alpha,a}}{\d t^\beta_b}-\left.\tOmega^{\alpha,a}_{\beta,b}\right|_{w^\gamma_c=w^{\top;\gamma}_c}\in\mbC[[t^*_*,\eps]]
$$
satisfies the condition
$$
\Coef_{\eps^{2g}}\left.\frac{\d^n\Omega^{\red;\alpha,a}_{\beta,b}}{\d t^{\alpha_1}_{d_1}\cdots\d t^{\alpha_n}_{d_n}}\right|_{t^*_*=0}=0\qquad\begin{minipage}{10cm}
for any $g,n\ge 0$ and $\alpha_1,\ldots,\alpha_n\in[N]$, $d_1,\ldots,d_n\in\mbZ_{\ge 0}$\\
satisfying $\sum d_i\le 2g$.
\end{minipage}
$$
We have the following equivalences:
\begin{gather}\label{eq:equivalences}
\Omega^{\alpha,a}_{\beta,b}\in\hcA_w\quad\Leftrightarrow\quad \Omega^{\alpha,a}_{\beta,b}=\tOmega^{\alpha,a}_{\beta,b}\quad\Leftrightarrow\quad\Omega^{\red;\alpha,a}_{\beta,b}=0.
\end{gather}

\medskip

In~\cite[Section~2.3]{BS22}, the authors defined cohomology classes
$$
B^2_{g,(d_1,\ldots,d_n)}\in H^{2\sum d_i}(\oM_{g,n+2}),\quad g\ge 0,\,n\ge 1,\,d_1,\ldots,d_n\in\mbZ_{\ge 0},
$$
and proved that 
\begin{gather}\label{eq:Omegared and B-class}
\Coef_{\eps^{2g}}\left.\frac{\d^n\Omega^{\red;\alpha,a}_{\beta,b}}{\d t^{\alpha_1}_{d_1}\cdots\d t^{\alpha_n}_{d_n}}\right|_{t^*_*=0}=\int_{\oM_{g,n+2}}B^2_{g,(d_1,\ldots,d_n)}c_{g,n+2}\left(\otimes_{i=1}^n e_{\alpha_i}\otimes \eta^{\alpha\mu} e_\mu\otimes e_\beta\right)\psi_{n+1}^a\psi_{n+2}^b,
\end{gather}
for any $g\ge 0$, $n\ge 1$, $d_1,\ldots,d_n\ge 0$, and $1\le\alpha_1,\ldots,\alpha_n\le N$ satisfying $\sum d_i\ge 2g+1$. In~\cite[Conjecture~1]{BS22}, the authors conjectured that $B^2_{g,(d_1,\ldots,d_n)}=0$ for arbitrary $g\ge 0$, $n\ge 1$, and $d_1,\ldots,d_n\ge 0$ such that $\sum d_i\ge 2g+1$. If this conjecture is true, then $\Omega^{\alpha,a}_{\beta,b}$ are differential polynomials for an arbitrary partial CohFT.

\medskip

\subsection{The double ramification hierarchy}\label{subsection:DR hierarchy}

Let us recall important properties of the double ramification cycle. The restriction $\left.\DR_g(a_1,\ldots,a_n)\right|_{\cM^{\ct}_{g,n}}\in H^{2g}(\cM^{\ct}_{g,n})$ depends polynomially on the $a_i$-s, where by $\cM^\ct_{g,n}\subset\oM_{g,n}$ we denote the moduli space of curves of compact type (see, e.g.,~\cite{JPPZ17}). This implies that the class $\lambda_g\DR_g(a_1,\ldots,a_n)\in H^{4g}(\oM_{g,n})$ depends polynomially on the $a_i$-s. Moreover, the resulting polynomial (with the coefficients in $H^{4g}(\oM_{g,n})$) is homogeneous of degree $2g$. The polynomiality of the class $\DR_g(a_1,\ldots,a_n)\in H^{2g}(\oM_{g,n})$ is proved by A. Pixton and D. Zagier (we thank A. Pixton for informing us about that), but the proof is not published~yet.

\medskip

Consider again an arbitrary partial CohFT of rank $N$. Let $u^1,\ldots,u^N$ be formal variables and consider the associated ring of differential polynomials $\hcA_u$. Define differential polynomials $P^\alpha_{\beta,d}\in\hcA_{u;0}$, $\alpha,\beta\in[N]$, $d\in\mbZ_{\ge 0}$, by
\begin{align*}
P^\alpha_{\beta,d}:=&\eta^{\alpha\mu}\sum_{\substack{g,n\geq 0\\k_1,\ldots,k_n\geq 0\\\sum_{j=1}^n k_j=2g}}\frac{\eps^{2g}}{n!}\prod_{j=1}^n u^{\alpha_j}_{k_j}\times\\
&\times\Coef_{(a_1)^{k_1}\cdots(a_n)^{k_n}} \left(\int_{\oM_{g,n+2}}\DR_g\bigg(-\sum_{j=1}^n a_j,0,a_1,\ldots,a_n\bigg)\lambda_g \psi_2^d c_{g,n+2}(e_\mu\otimes e_\beta\otimes \otimes_{j=1}^n e_{\alpha_j}) \right).
\end{align*}
The \emph{double ramification (DR) hierarchy}~\cite{Bur15,BDGR18} is the following system of evolutionary PDEs with one spatial variable:
\begin{gather*}
\frac{\d u^\alpha}{\d t^\beta_d}=\d_x P^\alpha_{\beta,d},\quad 1\le \alpha,\beta\in[N],\,d\in\mbZ_{\ge 0}.
\end{gather*}
In~\cite[Proposition~9.1]{BDGR18}, the authors proved that all the flows of the DR hierarchy are compatible with each other.

\medskip

Consider the solution $(u^1(x,t^*_*,\eps),\ldots,u^N(x,t^*_*,\eps))$ of the DR hierarchy specified by the initial condition $u^\alpha(x,t^*_*,\eps)|_{t^*_*=0}=A^\alpha x$. Since $P^\alpha_{\un,0}=u^\alpha$, this solution has the form $u^\alpha(x,t^*_*,\eps)=u^{\str;\alpha}|_{t^\gamma_0\mapsto t^\gamma_0+A^\gamma x}$ for some formal power series $u^{\str;\alpha}\in\mbC[[t^*_*,\eps]]$. The $N$-tuple $(u^{\str;1},\ldots,u^{\str;N})$ will be called the \emph{string solution} of the DR hierarchy. We denote $u^{\str;\alpha}_n:=\frac{\d^n u^{\str;\alpha}}{(\d t^\un_0)^n}$.

\medskip

\subsection{A relation between the two hierarchies}\label{subsection:relation between two hierarchies}

Consider an arbitrary partial CohFT. By \cite[Theorem~4.6]{BS22}, there exists a unique differential polynomial $\tOmega^{\alpha,a}\in\hcA_{w;-1}$ such that the difference
$$
\Omega^{\red;\alpha,a}:=\mcF^{\alpha,a}-\left.\tOmega^{\alpha,a}\right|_{w^\gamma_c=w^{\top;\gamma}_c}\in\mbC[[t^*_*,\eps]]
$$
satisfies the condition
$$
\Coef_{\eps^{2g}}\left.\frac{\d^n\Omega^{\red;\alpha,a}}{\d t^{\alpha_1}_{d_1}\cdots\d t^{\alpha_n}_{d_n}}\right|_{t^*_*=0}=0\qquad\begin{minipage}{10cm}
for any $g,n\ge 0$ and $\alpha_1,\ldots,\alpha_n\in[N]$, $d_1,\ldots,d_n\in\mbZ_{\ge 0}$\\
satisfying $\sum d_i\le 2g-1$.
\end{minipage}
$$

\medskip

Consider the DR hierarchy associated to our partial CohFT. Let us introduce~$N$ formal power series $\mcF^{\DR;\alpha}\in\mbC[[t^*_*,\eps]]$, $1\le\alpha\le N$, by the relation
$$
\frac{\d\mcF^{\DR;\alpha}}{\d t^\beta_b}:=\left.P^\alpha_{\beta,b}\right|_{u^\gamma_n=u^{\str;\gamma}_n},
$$
with the constant terms defined to be equal to zero, $\left.\mcF^{\DR;\alpha}\right|_{t^*_*=0}:=0$. Clearly, $\frac{\d\mcF^{\DR;\alpha}}{\d t^\un_0}=u^{\str;\alpha}$.

\medskip

By \cite[Theorems~2.2,~4.6,~4.9]{BS22}, we have
\begin{gather}\label{eq:one-point equality}
\left.\frac{\d\Omega^{\red;\alpha,0}}{\d t^\beta_b}\right|_{t^*_*=0}=\left.\frac{\d\mcF^{\DR;\alpha}}{\d t^\beta_b}\right|_{t^*_*=0},\quad \alpha,\beta\in [N],\,b\in\mbZ_{\ge 0}.
\end{gather}
Conjecturally~\cite[Section~4.4.3]{BS22}, we have a much stronger statement: $\Omega^{\red;\alpha,0}=\mcF^{\DR;\alpha}$. 

\medskip

\subsection{Reconstruction of a hierarchy from the equation of one flow}\label{subsection:reconstruction from one flow}

Consider an integer $N\ge 1$ and the algebra of differential polynomials $\hcA_w$ associated to formal variables $w^1,\ldots,w^N$.

\medskip

\begin{proposition}\label{proposition:reconstruction}
Consider a compatible system of PDEs
\begin{gather}\label{eq:system for reconstruction}
\frac{\d w^\alpha}{\d t^\beta_b}=\d_x P^\alpha_{\beta,b},\quad\alpha,\beta\in[N],\, b\in\mbZ_{\ge 0},\quad P^\alpha_{\beta,b}\in\hcA_{w;0},
\end{gather}
and a nonzero $N$-tuple $(A^1,\ldots,A^N)\in\mbC^N$ satisfying
$$
\left.P^{\alpha}_{\beta,b}\right|_{w^*_*=\eps=0}=0,\qquad
P^\alpha_{\un,0}=w^\alpha,\qquad
\frac{\d P^\alpha_{\beta,b}}{\d w^\un}=
\begin{cases}
P^\alpha_{\beta,b-1},&\text{if $b\ge 1$},\\
\delta^\alpha_\beta,&\text{if $b=0$},
\end{cases}
$$
where $P^\alpha_{\un,0}:=A^\gamma P^\alpha_{\gamma,0}$ and $\frac{\d}{\d w^\un}:=A^\gamma\frac{\d}{\d w^\gamma}$. Then all the differential polynomials $P^\alpha_{\beta,b}$ are uniquely determined by the $N$-tuple of differential polynomials $(P^1_{\un,1},\ldots P^N_{\un,1})$.
\end{proposition}

\medskip

The flow $\frac{\d}{\d t^\un_1}$ will be called the \emph{special flow} of the system~\eqref{eq:system for reconstruction}.

\medskip

\begin{proof}
The proof strongly uses the ideas from~\cite[Section~5.1 and~5.2]{BG16}. Consider the solution $(w^{\sp;1},\ldots,w^{\sp;N})$ of the system~\eqref{eq:system for reconstruction} given by the initial condition~$\left.w^{\sp;\alpha}\right|_{t^*_*=0}=A^\alpha x$. 

\medskip

\begin{lemma}
We have
\begin{align}
&\frac{\d w^{\sp;\alpha}}{\d t^\un_0}-\sum_{n\ge 0}t^\gamma_{n+1}\frac{\d w^{\sp;\alpha}}{\d t^\gamma_n}=A^\alpha,\label{eq:general string}\\
&\frac{\d w^{\sp;\alpha}}{\d t^\un_1}-\eps\frac{\d w^{\sp;\alpha}}{\d\eps}-x\frac{\d w^{\sp;\alpha}}{\d x}-\sum_{n\ge 0}t^\gamma_n\frac{\d w^{\sp;\alpha}}{\d t^\gamma_n}=0,\label{eq:general dilaton}\\
&\left.w^{\sp;\alpha}_n\right|_{x=0}=t^\alpha_n+\delta_{n,1}A^\alpha+Q^\alpha_n+O(\eps)\quad \text{for some $Q^\alpha_n\in\mbC[[t^*_*]]^{(n+1)}$},\label{eq:property of special}\\
&\left.\frac{\d w^{\sp;\alpha}}{\d t^\beta_b}\right|_{x=t^*_*=\eps=0}=\delta^\alpha_\beta\delta_{b,0}.\label{eq:simplest coefficient of special}
\end{align}
\end{lemma}
\begin{proof}
Consider formal variables $v^\alpha$, $\alpha\in[N]$, together with the associated variables $v^\alpha_d$ with $d\ge 0$, and let us relate them to the variables $w^\beta_q$ by $w^\alpha_d=v^\alpha_{d+1}$. From the compatibility of the system~\eqref{eq:system for reconstruction} it follows that the system
\begin{gather*}
\frac{\d v^\alpha}{\d t^\beta_b}=P^\alpha_{\beta,b},\quad\alpha,\beta\in[N],\, b\in\mbZ_{\ge 0},
\end{gather*}
is also compatible. We consider its solution $(v^{\sp;1},\ldots,v^{\sp;N})$ given by the initial condition~$\left.v^{\sp;\alpha}\right|_{t^*_*=0}=A^\alpha \frac{x^2}{2}$. Clearly $\d_x v^{\sp;\alpha}=\frac{\d v^{\sp;\alpha}}{\d t^\un_0}=w^{\sp;\alpha}$. 

\medskip

We claim that 
\begin{gather}\label{eq:string for v}
\frac{\d v^{\sp;\alpha}}{\d t^\un_0}-\sum_{n\ge 0}t^\gamma_{n+1}\frac{\d v^{\sp;\alpha}}{\d t^\gamma_n}=t^\alpha_0+A^\alpha x,
\end{gather}
this is proved similarly to Lemma~4.7 from~\cite{Bur15}. Applying $\d_x$ to this equation, we obtain equation~\eqref{eq:general string}. Note that equation~\eqref{eq:string for v} implies that $\left.w^{\sp;\alpha}\right|_{t^*_{\ge 1}=0}=t^\alpha_0+A^\alpha x$, and therefore equation~\eqref{eq:property of special} is true for $n=0$. Equation~\eqref{eq:general string} implies that
$$
\left.w^{\sp;\alpha}_n\right|_{x=0}=\delta_{n,1}A^\alpha+\left(\sum_{k\ge 0}t^\gamma_{k+1}\frac{\d}{\d t^\gamma_k}\right)^n(w^{\sp;\alpha}|_{x=0}),
$$
and therefore equation~\eqref{eq:property of special} is true for all $n\ge 0$.

\medskip

Equation~\eqref{eq:general dilaton} is proved similarly to Proposition 5.1 from~\cite{BG16}.

\medskip

To prove~\eqref{eq:simplest coefficient of special}, we use~\eqref{eq:property of special} and obtain
$$
\left.\frac{\d w^{\sp;\alpha}}{\d t^\beta_b}\right|_{x=t^*_*=\eps=0}=\left.\left(\left.\d_x P^\alpha_{\beta,b}\right|_{w^\gamma_r=w^{\sp;\gamma}_r}\right)\right|_{x=t^*_*=\eps=0}=\left.\frac{\d P^\alpha_{\beta,b}}{\d w^\gamma}\right|_{w^*_*=\eps=0}A^\gamma=\left.\frac{\d P^\alpha_{\beta,b}}{\d w^\un}\right|_{w^*_*=\eps=0}=\delta^\alpha_\beta\delta_{b,0}.
$$
\end{proof}

\medskip

The property~\eqref{eq:property of special} implies that all the differential polynomials $P^\alpha_{\beta,b}$ are uniquely determined by the solution $(w^{\sp;1},\ldots,w^{\sp;N})$ of the system~\eqref{eq:system for reconstruction} (see the discussion in Section~\ref{subsection:DZ hierarchy}). So it remains to prove that this solution is uniquely determined by the differential polynomials $P^\alpha_{\un,1}$, $1\le\alpha\le N$. Since $\frac{\d w^{\sp;\alpha}}{\d x}=\frac{\d w^{\sp;\alpha}}{\d t^\un_0}$, it is sufficient to determine the formal power series $\left.w^{\sp;\alpha}\right|_{x=0}$. 

\medskip

Denote by $c^{\alpha;\alpha_1,\ldots,\alpha_n}_{k;d_1,\ldots,d_n}$ the coefficient of $\eps^k t^{\alpha_1}_{d_1}\cdots t^{\alpha_n}_{d_n}$ in $\left.w^{\sp;\alpha}\right|_{x=0}$. We will call these coefficients \emph{$c$-coefficients}. We will say that the multi-index of the coefficient $c^{\alpha;\alpha_1,\ldots,\alpha_n}_{k;d_1,\ldots,d_n}$, comparing with the multi-index of another coefficient~$c^{\beta;\beta_1,\ldots,\beta_m}_{l;q_1,\ldots,q_m}$,
\begin{itemize}
\item has \emph{smaller order} if $k<l$;

\smallskip

\item has \emph{smaller length} if $n<m$;

\smallskip

\item has \emph{smaller weight} if $\sum d_i<\sum q_j$;

\smallskip

\item is \emph{smaller} if $k<l$ or $(k=l\text{ and }n<m)$ or $\left(k=l, n=m,\text{ and }\sum d_i<\sum q_j\right)$.
\end{itemize}

\medskip

Equation~\eqref{eq:general dilaton} together with the equations for the flow $\frac{\d}{\d t^\un_1}$ imply that
\begin{gather}\label{eq:dilaton plus special flow}
\left(\eps\frac{\d}{\d\eps}+\sum_{r\ge 0}t^\gamma_r\frac{\d}{\d t^\gamma_r}\right)\left(\left.w^{\sp;\alpha}\right|_{x=0}\right)=\left.\left(\left.\d_x P^\alpha_{\un,1}\right|_{w^\gamma_r=w^{\sp;\gamma}_r}\right)\right|_{x=0}.
\end{gather}
Let us show that this equation allows to compute all the $c$-coefficients recursively. Indeed, the coefficient of $\eps^k t^{\alpha_1}_{d_1}\cdots t^{\alpha_n}_{d_n}$ on the left-hand side of~\eqref{eq:dilaton plus special flow} is equal to $(k+n)c^{\alpha;\alpha_1,\ldots,\alpha_n}_{k;d_1,\ldots,d_n}$. Let us look now at the coefficient of $\eps^k t^{\alpha_1}_{d_1}\cdots t^{\alpha_n}_{d_n}$ on the right-hand side of~\eqref{eq:dilaton plus special flow}. Note that the coefficient of $\eps^k t^{\alpha_1}_{d_1}\cdots t^{\alpha_n}_{d_n}$ in
\begin{enumerate}
\item $\left.w^{\sp;\alpha}_l\right|_{x=0}$, $l>0$, is a linear combination of the $c$-coefficients with multi-index of order $k$ and length $n$, but with weight less than $\sum d_i$ (this follows from equation~\eqref{eq:general string});

\smallskip

\item $\left.\left(\eps^l w^{\sp;\beta_1}_{q_1}\cdots w^{\sp;\beta_m}_{q_m}\right)\right|_{x=0}$, $l>0$, is a polynomial in the $c$-coefficients with multi-index of order less than $k$;

\smallskip

\item $\left.\left(w^{\sp;\beta_1}\cdots w^{\sp;\beta_m}w^{\sp;\beta}_x\right)\right|_{x=0}$, $m\ge 2$, is a polynomial in the $c$-coefficients with multi-index of order less than $k$, or of order $k$ but length less than $n$;

\smallskip

\item $\left.\left(w^{\sp;\beta_1}w^{\sp;\beta}_x\right)\right|_{x=0}$, is equal to $A^\beta c^{\beta_1;\alpha_1,\ldots,\alpha_n}_{k;d_1,\ldots,d_n}$ plus a polynomial in the $c$-coefficients with multi-index of order less than $k$, or of order $k$ but length less than $n$.
\end{enumerate}
Since $\frac{\d P^\alpha_{\un,1}}{\d w^\gamma}A^\gamma=\frac{\d P^\alpha_{\un,1}}{\d w^\un}=P^\alpha_{\un,0}=w^\alpha$, we conclude that the coefficient of $\eps^k t^{\alpha_1}_{d_1}\cdots t^{\alpha_n}_{d_n}$ on the right-hand side of~\eqref{eq:dilaton plus special flow} is equal to $c^{\alpha;\alpha_1,\ldots,\alpha_n}_{k;d_1,\ldots,d_n}$ plus a polynomial in the $c$-coefficients with smaller multi-index. Therefore, equation~\eqref{eq:dilaton plus special flow} allows to compute the coefficient $c^{\alpha;\alpha_1,\ldots,\alpha_n}_{k;d_1,\ldots,d_n}$ in terms of the $c$-coefficients with smaller multi-index unless $k+n=1$. It remains to note that we already know these exceptional coefficients: $c^\alpha_1=0$ and $c^{\alpha;\alpha_1}_{0;d_1}=\delta^{\alpha,\alpha_1}\delta_{d_1,0}$.
\end{proof}

\medskip

Note that the DR hierarchy corresponding to an arbitrary partial CohFT satisfies the assumptions of the proposition. The same is true for the DZ hierarchy, if $\Omega^{\alpha,0}_{\beta,b}$ are differential polynomials.

\medskip

\subsection{Comments about the infinite rank case}\label{subsection:comments about infinite rank}

As it was already explained in~\cite{BR21a,BR22}, a notion of infinite rank partial CohFT (i.e. a partial CohFT with an infinite dimensional phase space~$V$) requires some care, because one needs to clarify what is meant by the matrix~$(\eta^{\alpha\beta})$ and to make sense of the, a priori infinite, sum over $\mu$ and $\nu$, both appearing in Axiom (iii). One way to give a rigorous definition is the following. Consider a vector space~$V$ with basis labeled by integers, $V=\mathrm{span}(\{e_\alpha\}_{\alpha\in\mbZ})$, and suppose that for any~$(g,n)$ in the stable range and each $\alpha_1,\ldots,\alpha_{n-1} \in \mbZ$ the set $\{\beta\in\mbZ\, |\, c_{g,n}(\otimes_{i=1}^{n-1} e_{\alpha_i}\otimes e_\beta)\neq 0\}$ is finite. In particular, this implies that the matrix $(\eta_{\alpha\beta})$ is row- and column-finite (each row and each column have a finite number of nonzero entries), which is equivalent to $\eta^\sharp(V)\subseteq \mathrm{span}(\{e^\alpha\}_{\alpha \in \mbZ})$, where $\eta^\sharp\colon V\to V^*$ is the injective map induced by the bilinear form $\eta$ and $\{e^\alpha\}_{\alpha \in \mbZ}$ is the dual ``basis''. Let us further demand that the injective map $\eta^\sharp\colon V \to \mathrm{span}(\{e^\alpha\}_{\alpha \in \mbZ})$ is surjective too, i.e. that a unique two-sided row- and column-finite matrix $(\eta^{\alpha\beta})$, inverse to $(\eta_{\alpha\beta})$, exists (it represents the inverse map $(\eta^\sharp)^{-1}\colon\mathrm{span}(\{e^\alpha\}_{\alpha \in \mbZ})\to V$). Then the equation appearing in Axiom (iii) is well defined with the double sum only having a finite number of nonzero terms. Such a partial CohFT will be called a {\it tame partial CohFT of infinite rank}.

\medskip

Note that if a tame partial CohFT of infinite rank satisfies the condition
\begin{gather}\label{eq:condition for partial CohFT}
\begin{minipage}{10cm}
for any $(g,n)$ in the stable range and $\alpha_1,\ldots,\alpha_n\in\mbZ$, the class $c_{g,n}(\otimes_{i=1}^n e_{\alpha_i})$ is zero unless $\sum\alpha_i=0$,
\end{minipage}
\end{gather}
then $e=\lambda e_0$ for some $\lambda\in\mbC^*$, the matrices $(\eta_{\alpha\beta})$ and $(\eta^{\alpha\beta})$ are antidiagonal, and the double sum on the right-hand side of the equation in Axiom (iii) has at most one nonzero term.

\medskip

Let us now make comments regarding the theory of evolutionary PDEs with infinitely many dependent variables. By a \emph{differential polynomial in infinitely many variables} $w^\alpha$, $\alpha\in\mbZ$, we mean an element $f\in\mbC[[w^*_*]]$ consisting of an infinite linear combination of monomials $w^{\alpha_1}_{d_1}\cdots w^{\alpha_n}_{d_n}$ such that $\sum d_i\le m$ for some $m\in\mbZ_{\ge 0}$. The algebra of such differential polynomials will be denoted by $\cA_w$. Let $\hcA_w:=\cA_w[[\eps]]$. As in the case of finitely many dependent variables, we can consider systems of evolutionary PDEs $\frac{\d w^\alpha}{\d t}=P^\alpha$, $\alpha\in\mbZ$, $P^\alpha\in\hcA_w$. However, in general, defining the compatibility of this system with another system $\frac{\d w^\alpha}{\d s}=Q^\alpha$ can be problematic. Indeed, consider the example $P^\alpha=\sum_{\gamma\in\mbZ}w^\gamma$ and $Q^\beta=w^0$. Then we see that the sum $\sum_{\substack{\beta\in\mbZ\\n\in\mbZ_{\ge 0}}}\left(\frac{\d P^\alpha}{\d w^\beta_n}\d_x^n Q^\beta-\frac{\d Q^\alpha}{\d w^\beta_n}\d_x^n P^\beta\right)$ is not well defined.

\medskip

In order to overcome this problem, let us consider systems of PDEs of special form. Let us introduce an additional grading $\tdeg$ on~$\hcA_w$ by $\tdeg w^\alpha_d:=\alpha$ and $\tdeg\eps:=0$. We will say that a system of evolutionary PDEs  $\frac{\d w^\alpha}{\d t}=P^\alpha$ is \emph{homogeneous} if $\tdeg P^\alpha=\alpha$. It is easy to check that the notion of compatible homogeneous system of PDEs is well defined. Note that the noncommutative KdV hierarchy, viewed as a system of evolutionary PDEs with one spatial variable and infinitely many dependent variables, is homogeneous.

\medskip

Consider now a tame partial CohFT of infinite rank satisfying~\eqref{eq:condition for partial CohFT}. Then it is easy to check that the DZ and DR hierarchies are well defined for such a partial CohFT. Moreover, the DR hierarchy is homogeneous, the DZ hierarchy is also homogeneous if $\Omega^{\alpha,0}_{\beta,b}$ are differential polynomials, and all the results from Sections~\ref{subsection:DZ hierarchy}--\ref{subsection:relation between two hierarchies} remain true. Regarding Proposition~\ref{proposition:reconstruction}, it remains true, with the same proof, if we require that the system under consideration is homogeneous and $A^\alpha=\delta^{\alpha,0}$.

%%%%%%%%%%%%%%%%%%%%%%%%%%%%%%%%%%%%%%%%%%%%%%%%%%%%%%%%%%%%%%%%%%%%
%%%%%%%%%%%%%%%%%%%%%%%%%%%%%%%%%%%%%%%%%%%%%%%%%%%%%%%%%%%%%%%%%%%%

\section{Proof of Theorem~\ref{theorem:main}}\label{section:proof of the main theorem}

Let $V$ be a vector space with basis $e_a$ indexed by integers $a\in\mbZ$. Consider a family of linear maps
$$
c^\P_{g,n}\colon V^{\otimes n}\to H^{\even}(\oM_{g,n})\otimes\mbC[\mu]
$$
given by
$$
c^\P_{g,n}\left(\otimes_{i=1}^n e_{a_i}\right):=\sum_{d=0}^g 2^{-d}\mu^{2d}P_g^d(a_1,\ldots,a_n),\quad a_1,\ldots,a_n\in\mbZ.
$$
By~\cite[Proposition~4.6]{BR22}, the maps $c_{g,n}^\P$ form a tame partial CohFT of infinite rank with the phase space~$V$, the unit~$e_0$, and the metric given in the basis $\{e_a\}_{a\in\mbZ}$ by $\eta_{ab}=\delta_{a+b,0}$. This partial CohFT satisfies the condition~\eqref{eq:condition for partial CohFT}, and therefore, as it was explained in Section~\ref{subsection:comments about infinite rank}, we can use all the results from Sections~\ref{subsection:DZ hierarchy}--\ref{subsection:reconstruction from one flow}.

\medskip

\subsection{Using the computation of the DR hierarchy}\label{subsection:using DR hierarchy}

For the partial CohFT $\{c_{g,n}^P\}$, we consider the associated DR and DZ hierarchies.

\begin{proposition}\label{proposition:main}
We have
\begin{itemize}
\item[1.] $\Omega^{\alpha,a}_{\beta,b}\in\hcA_{w;0}$,

\smallskip

\item[2.] $\Omega^{\alpha,0}_{0,1}=\frac{1}{2}\sum_{\alpha_1+\alpha_2=\alpha}\frac{1}{T_\alpha(\eps\mu\d_x)}\left(\left[e^{-\frac{\alpha_2}{2}\eps\mu\d_x}T_{\alpha_1}(\eps\mu\d_x)w^{\alpha_1}\right]\cdot\left[e^{\frac{\alpha_1}{2}\eps\mu\d_x}T_{\alpha_2}(\eps\mu\d_x)w^{\alpha_2}\right]\right)+\frac{\eps^2}{12}w^\alpha_{xx}$.
\end{itemize}
\end{proposition}

\medskip

Before proving the proposition, let us show how to prove Theorem~\ref{theorem:main} using it. Indeed, the collection of formal power series $(w^{\P;\alpha})_{\alpha\in\mbZ}$ is the topological solution of the DZ hierarchy. Consider the Miura transformation
$$
w^\alpha\mapsto u^\alpha(w^*_*,\mu,\eps):=T_\alpha(\eps\mu\d_x)w^\alpha.
$$
After this transformation, the DZ hierarchy has the form
$$
\frac{\d u^\alpha}{\d t^\beta_b}=\d_x\tP^\alpha_{\beta,b},\qquad \tP^\alpha_{\beta,b}\in\hcA_{u;0}.
$$
We see that it is sufficient to check that the system of flows $\frac{\d}{\d t^0_d}$, $d\ge 1$, of the transformed DZ hierarchy coincides with the noncommutative KdV hierarchy, viewed as a system of evolutionary PDEs with one spatial variable, see Remark~\ref{remark:about ncKdV}. By \cite[Theorem~4.1]{BR21a}, the system of flows $\frac{\d}{\d t^0_d}$, $d\ge 1$, of the DR hierarchy coincides with the noncommutative KdV hierarchy. Therefore, it is sufficient to check that $P^\alpha_{\beta,b}=\tP^\alpha_{\beta,b}$. By Part 2 of Proposition~\ref{proposition:main} and formula~\eqref{eq:ncKdV,flow1,with x}, we have $P^\alpha_{0,1}=\tP^\alpha_{0,1}$. Then Proposition~\ref{proposition:reconstruction} implies that $P^\alpha_{\beta,b}=\tP^\alpha_{\beta,b}$, which completes the proof of the Theorem~\ref{theorem:main}. 

\medskip

The two parts of Proposition~\ref{proposition:main} will be proved in the next two sections, respectively.

\medskip

\subsection{Proof of Proposition~\ref{proposition:main}: Part 1}\label{subsection:part 1}

By~\eqref{eq:equivalences}, Part 1 of the proposition is equivalent to the vanishing $\Omega^{\red;\alpha,a}_{\beta,b}=0$. By~\eqref{eq:Omegared and B-class}, this is equivalent to the vanishing
\begin{gather}\label{eq:needed vanishing}
\int_{\oM_{g,n+2}}B^2_{g,(d_1,\ldots,d_n)}c^P_{g,n+2}\left(\otimes_{i=1}^n e_{\alpha_i}\otimes e_{-\alpha}\otimes e_\beta\right)\psi_{n+1}^a\psi_{n+2}^b=0,
\end{gather}
for any $g\ge 0$, $n\ge 1$, $d_1,\ldots,d_n\in\mbZ_{\ge 0}$, and $\alpha_1,\ldots,\alpha_n\in\mbZ$ satisfying $\sum d_i\ge 2g+1$ and $\alpha=\beta+\sum\alpha_i$.

\medskip

In~\cite[proof of Proposition~4.6]{BR22}, for arbitrary $r\ge 1$, the authors constructed a certain semisimple CohFT $\{\Omega^r_{g,n}\}$ of rank $r$, with phase space $V_r=\mathrm{span}(\{e_0,\ldots,e_{r-1}\})$, metric $\eta_r(e_a,e_b)=\frac{1}{r}\delta_{a+b=0\text{ mod } r}$, and unit $e_0$. Since the CohFT $\{\Omega^r_{g,n}\}$ is semisimple, by~\cite[Theorem~22]{BPS12} and~\eqref{eq:equivalences}, we have
\begin{gather}\label{eq:vanishing with Omegar}
\int_{\oM_{g,n+2}}B^2_{g,(d_1,\ldots,d_n)}\left(r^{1-2g}\Omega^r_{g,n+2}\left(\otimes_{i=1}^n e_{\widetilde{\alpha_i}}\otimes e_{\widetilde{-\alpha}}\otimes e_{\widetilde{\beta}}\right)\right)\psi_{n+1}^a\psi_{n+2}^b=0,
\end{gather}
where for an integer $\gamma$ we denote by $\widetilde{\gamma}\in\{0,\ldots,r-1\}$ the unique number such that $\gamma=\widetilde{\gamma}\text{ mod }r$. By~\cite[proof of Proposition~4.6]{BR22}, for fixed $g,n,\alpha,\beta,\alpha_1,\ldots,\alpha_n$, the class 
$$
r^{1-2g}\Omega^r_{g,n+2}\left(\otimes_{i=1}^n e_{\widetilde{\alpha_i}}\otimes e_{\widetilde{-\alpha}}\otimes e_{\widetilde{\beta}}\right)
$$
is a polynomial in $r$ for $r$ sufficiently large, and moreover the constant term of this polynomial is equal to the class $c^P_{g,n+2}\left(\otimes_{i=1}^n e_{\alpha_i}\otimes e_{-\alpha}\otimes e_\beta\right)$. Thus, the vanishing~\eqref{eq:vanishing with Omegar} implies the vanishing~\eqref{eq:needed vanishing}. This completes the proof of Part 1 of the proposition. 

\medskip

\subsection{Proof of Proposition~\ref{proposition:main}: Part 2}\label{subsection:part 2}

Since $\deg P^j_g(a_1,\ldots,a_n)=2j$, dimension counting gives
$$
\underbrace{\left(\frac{3}{2}\eps\frac{\d}{\d\eps}-\frac{1}{2}\mu\frac{\d}{\d\mu}+\sum(1-d)t^\gamma_d\frac{\d}{\d t^\gamma_d}\right)}_{=:L}\mcF^\P=3\mcF^\P,
$$
which implies that
$$
L\left(\frac{\d^2\mcF^\P}{\d t^\alpha_p\d t^\beta_q}\right)=(1+p+q)\frac{\d^2\mcF^\P}{\d t^\alpha_p\d t^\beta_q},\qquad L\left(w^{\P;\alpha}_n\right)=(1-n)w^{\P;\alpha}_n.
$$
Therefore,
\begin{gather}\label{eq:degree condition-0}
\left(\frac{3}{2}\eps\frac{\d}{\d\eps}-\frac{1}{2}\mu\frac{\d}{\d\mu}+\sum_{n\ge 0} (1-n)w^\gamma_n\frac{\d}{\d w^\gamma_n}\right)\Omega^{\alpha,p}_{\beta,q}=(1+p+q)\Omega^{\alpha,p}_{\beta,q}.
\end{gather}
Since $\Omega^{\alpha,a}_{\beta,b}\in\hcA_{w;0}$, we have $\left(\eps\frac{\d}{\d\eps}-\sum_{n\ge 0}n w^\gamma_n\frac{\d}{\d w^\gamma_n}\right)\Omega^{\alpha,p}_{\beta,q}=0$, and combining this with~\eqref{eq:degree condition-0} we obtain
\begin{gather}\label{eq:degree condition}
\left(\frac{1}{2}\eps\frac{\d}{\d\eps}-\frac{1}{2}\mu\frac{\d}{\d\mu}+\sum_{n\ge 0}w^\gamma_n\frac{\d}{\d w^\gamma_n}\right)\Omega^{\alpha,p}_{\beta,q}=(1+p+q)\Omega^{\alpha,p}_{\beta,q}.
\end{gather}

\medskip

Note that $\mcF^\P\in\mbC[[t^*_*,\eps^2,(\eps\mu)^2]]$, which implies that the coefficient of $\eps^i\mu^j$ in $\Omega^{\alpha,p}_{\beta,q}$ is zero unless $i,j$ are even and $i\ge j$. Therefore, $\Omega^{\alpha,0}_{0,1}$ has the form
\begin{gather}\label{eq:general formula for Omega}
\Omega^{\alpha,0}_{0,1}=\sum_{g\ge 0}\frac{(\eps\mu)^{2g}}{2}\sum_{\substack{\beta+\gamma=\alpha\\k_1+k_2=2g}}A_{\beta,k_1;\gamma,k_2}w^\beta_{k_1}w^\gamma_{k_2}+\sum_{g\ge 1}\eps^{2g}\mu^{2g-2}B_{\alpha,g} w^\alpha_{2g},\quad A_{\beta,k_1;\gamma,k_2},B_{\alpha,g}\in\mbC,
\end{gather}
where $A_{\beta,k_1;\gamma,k_2}=A_{\gamma,k_2;\beta,k_1}$.

\medskip

It remains to prove that 
\begin{align}
&A_{\alpha_1,k_1;\alpha_2,k_2}=\Coef_{z_1^{k_1}z_2^{k_2}}\left[\frac{T_{\alpha_1}(z_1)T_{\alpha_2}(z_2)}{T_{\alpha_1+\alpha_2}(z_1+z_2)}\frac{\ozeta(\alpha_1 z_2-\alpha_2 z_1)}{2}\right],\label{eq:A-coef}\\
&B_{\alpha,g}=\frac{\delta_{g,1}}{12}.\label{eq:B-coef}
\end{align}
where $\ozeta:=e^{z/2}+e^{-z/2}$.

\medskip

Let us prove equation~\eqref{eq:A-coef} or equivalently that
\begin{gather}\label{eq:A-coef-2}
A_{\alpha_1,\alpha_2}(z_1,z_2):=\sum_{k_1,k_2\ge 0}A_{\alpha_1,k_1;\alpha_2,k_2}z_1^{k_1}z_2^{k_2}=\frac{T_{\alpha_1}(z_1)T_{\alpha_2}(z_2)}{T_{\alpha_1+\alpha_2}(z_1+z_2)}\frac{\ozeta(\alpha_1 z_2-\alpha_2 z_1)}{2},
\end{gather}
where we define $A_{\alpha_1,k_1;\alpha_2,k_2}:=0$ when $k_1+k_2$ is not even. Introduce the following notations:
\begin{align*}
&C_{\alpha,g}:=\int_{\oM_{g,3}}\DR_g(\alpha,-\alpha,0)\psi_1^{2g},\\
&D^n_{\alpha_1,d_1;\alpha_2,d_2}:=
\begin{cases}
\int_{\oM_{g,n+3}}\DR_g(\alpha_1,\alpha_2,-\alpha_1-\alpha_2,\underbrace{0,\ldots,0}_{\text{$n$ times}})\psi_1^{d_1}\psi_2^{d_2},&\text{if $g:=\frac{d_1+d_2-n}{2}\in\mbZ_{\ge 0}$},\\
0,&\text{otherwise}.
\end{cases}
\end{align*}
By \cite[Theorems~1 and~2]{BSSZ15} (see also~\cite[Lemma~2.1]{Bur17}), we have
\begin{gather*}
\sum_{g\ge 0}C_{g,\alpha}z^{2g}=\frac{1}{T_{\alpha}(z)},\qquad \sum_{d_1,d_2\ge 0}D^n_{\alpha_1,d_1;\alpha_2,d_2}z_1^{d_1}z_2^{d_2}=(z_1+z_2)^n\frac{S(\alpha_1 z_2-\alpha_2 z_1)}{T_{\alpha_1+\alpha_2}(z_1+z_2)}.
\end{gather*}

\medskip

Let us fix $\alpha_1,\alpha_2\in\mbZ$ and $\alpha=\alpha_1+\alpha_2$. On both sides of~\eqref{eq:general formula for Omega}, let us now substitute $w^\gamma_n=w^{\P;\gamma}_n$, $\mu=\eps^{-1}$, apply the operator $\sum_{d_1,d_2\ge 0}z_1^{d_1}z_2^{d_2}\frac{\d^2}{\d t^{\alpha_1}_{d_1}\d t^{\alpha_2}_{d_2}}$, and finally set $\eps=t^*_*=0$. On the left-hand side, we obtain 
\begin{align*}
&\sum_{g\ge 0}\sum_{\substack{d_1,d_2\ge 0\\d_1+d_2=2g}}\left(\int_{\oM_{g,4}}\DR_g(\alpha_1,\alpha_2,0,-\alpha)\psi_1^{d_1}\psi_2^{d_2}\psi_3\right)z_1^{d_1}z_2^{d_2}=\\
&\hspace{4cm}=\sum_{d_1,d_2\ge 0}(d_1+d_2+1)D^0_{\alpha_1,d_1;\alpha_2,d_2}z_1^{d_1}z_2^{d_2}=\\
&\hspace{4cm}=\left(1+z_1\frac{\d}{\d z_1}+z_2\frac{\d}{\d z_2}\right)\frac{S(\alpha_1 z_2-\alpha_2 z_1)}{T_{\alpha_1+\alpha_2}(z_1+z_2)}.
\end{align*}

\medskip

Consider now the right-hand side of~\eqref{eq:general formula for Omega}. Let us first substitute $w^\gamma_n=w^{\P;\gamma}_n$, $\mu=\eps^{-1}$, and set $\eps=0$. We obtain
\begin{gather}\label{eq:simpler RHS-1}
\frac{1}{2}\sum_{\substack{\beta+\gamma=\alpha\\k_1,k_2\ge 0}}A_{\beta,k_1;\gamma,k_2}\tw^{\P;\beta}_{k_1}\tw^{\P;\gamma}_{k_2},\quad\text{where}\quad \tw^{\P;\beta}_d:=\left.\left(\left.w^{\P;\beta}_d\right|_{\mu=\eps^{-1}}\right)\right|_{\eps=0}.
\end{gather}
Let us now apply the operator $\sum_{d_1,d_2\ge 0}z_1^{d_1}z_2^{d_2}\frac{\d^2}{\d t^{\alpha_1}_{d_1}\d t^{\alpha_2}_{d_2}}$ to~\eqref{eq:simpler RHS-1} and set $t^*_*=0$, we obtain
\begin{gather}\label{eq:simpler RHS-2}
\sum_{d_1,d_2\ge 2}z_1^{d_1}z_2^{d_2}\sum_{\substack{\beta+\gamma=\alpha\\k_1,k_2\ge 0}}A_{\beta,k_1;\gamma,k_2}\left.\left(\frac{\d\tw^{\P;\beta}_{k_1}}{\d t^{\alpha_1}_{d_1}}\frac{\d\tw^{\P;\gamma}_{k_2}}{\d t^{\alpha_2}_{d_2}}\right)\right|_{t^*_*=0}+\sum_{d_1,d_2\ge 0}z_1^{d_1}z_2^{d_2}\sum_{h\ge 1}A_{0,1;\alpha,2h-1}\left.\left(\frac{\d^2\tw^{\P;\alpha}_{2h-1}}{\d t^{\alpha_1}_{d_1}\d t^{\alpha_2}_{d_2}}\right)\right|_{t^*_*=0},
\end{gather}
where we used that $\left.w^{\P;\beta}_n\right|_{t^*_*=0}=\delta^{\beta,0}\delta_{n,1}$. Note that
\begin{gather*}
\left.\frac{\d\tw^{\P;\beta}_k}{\d t^\alpha_d}\right|_{t^*_*=0}=
\begin{cases}
\delta_{\alpha,\beta}C_{\alpha,\frac{d-k}{2}},&\text{if $\frac{d-k}{2}\in\mbZ_{\ge 0}$},\\
0,&\text{otherwise},
\end{cases}
\qquad
\left.\frac{\d^2\tw^{\P;\alpha}_{2h-1}}{\d t^{\alpha_1}_{d_1}\d t^{\alpha_2}_{d_2}}\right|_{t^*_*=0}=D^{2h}_{\alpha_1,d_1;\alpha_2,d_2}.
\end{gather*}
Therefore, the formal power series~\eqref{eq:simpler RHS-2} is equal to
$$
\frac{A_{\alpha_1,\alpha_2}(z_1,z_2)}{T_{\alpha_1}(z_1)T_{\alpha_2}(z_2)}+\frac{S(\alpha_1z_2-\alpha_2z_1)}{T_{\alpha}(z_1+z_2)}\sum_{h\ge 1}A_{0,1;\alpha,2h-1}(z_1+z_2)^{2h},
$$
and we obtain the equation
\begin{multline}\label{eq:equation of formal power series}
\left(1+z_1\frac{\d}{\d z_1}+z_2\frac{\d}{\d z_2}\right)\frac{S(\alpha_1 z_2-\alpha_2 z_1)}{T_{\alpha}(z_1+z_2)}=\\
=\frac{A_{\alpha_1,\alpha_2}(z_1,z_2)}{T_{\alpha_1}(z_1)T_{\alpha_2}(z_2)}+\frac{S(\alpha_1z_2-\alpha_2z_1)}{T_{\alpha}(z_1+z_2)}\sum_{h\ge 1}A_{0,1;\alpha,2h-1}(z_1+z_2)^{2h}.
\end{multline}

\medskip

We claim that the system of equations~\eqref{eq:equation of formal power series} for all $\alpha_1,\alpha_2\in\mbZ$ determines all the numbers~$A_{\alpha_1,d_1;\alpha_2,d_2}$. Indeed, let us choose $\alpha_1=0$. Then the coefficient of $z_1^{d_1}z_2^{d_2}$ on the right-hand side of~\eqref{eq:equation of formal power series} is equal to $A_{0,d_1;\alpha,d_2}+{d_1+d_2\choose d_1}A_{0,1;\alpha,d_1+d_2-1}$ plus a linear combination of the numbers~$A_{0,k_1;\alpha,k_2}$ with $k_1+k_2<d_1+d_2=:d$. Suppose that all the numbers $A_{0,k_1;\alpha,k_2}$ with $k_1+k_2<d$ are already determined. Note that for $d_1=1$ we have $A_{0,d_1;\alpha,d_2}+{d_1+d_2\choose d_1}A_{0,1;\alpha,d_1+d_2-1}=(d_2+2)A_{0,1;\alpha,d_2}$. So it is clear that considering the coefficient of $z_1 z_2^{d-1}$ determines the number~$A_{0,1;\alpha,d-1}$. Then considering the coefficient of $z_1^{d_1}z_2^{d_2}$ for arbitrary $d_1,d_2$ satisfying $d_1+d_2=d$ determines the number~$A_{0,d_1;\alpha,d_2}$. Therefore, the system of equations~\eqref{eq:equation of formal power series} with $\alpha_1=0$ and arbitrary $\alpha_2$ determines all the numbers $A_{0,d_1;\alpha,d_2}$, or equivalently this system determines the second summand on the right-hand side of~\eqref{eq:equation of formal power series}. It becomes obvious now that the full system of equations~\eqref{eq:equation of formal power series} with $\alpha_1,\alpha_2\in\mbZ$ determines the first summand on the right-hand side of~\eqref{eq:equation of formal power series} and hence the formal power series $A_{\alpha_1,\alpha_2}(z_1,z_2)$ is determined.

\medskip

So in order to prove~\eqref{eq:A-coef-2}, it is sufficient to check that the formal power series 
$$
\tA_{\alpha_1,\alpha_2}(z_1,z_2):=\frac{T_{\alpha_1}(z_1)T_{\alpha_2}(z_2)}{T_{\alpha}(z_1+z_2)}\frac{\ozeta(\alpha_1 z_2-\alpha_2 z_1)}{2}
$$
satisfies equation~\eqref{eq:equation of formal power series}. Therefore, we have to check that
\begin{multline}\label{eq:equation with tA}
\left(1+z_1\frac{\d}{\d z_1}+z_2\frac{\d}{\d z_2}\right)\frac{S(\alpha_1 z_2-\alpha_2 z_1)}{T_{\alpha}(z_1+z_2)}=\\
=\frac{\ozeta(\alpha_1 z_2-\alpha_2 z_1)}{2 T_{\alpha}(z_1+z_2)}+\frac{S(\alpha_1z_2-\alpha_2z_1)}{T_{\alpha}(z_1+z_2)}\sum_{h\ge 1}\tA_{0,1;\alpha,2h-1}(z_1+z_2)^{2h}.
\end{multline}
Note that $\left(1+z_1\frac{\d}{\d z_1}+z_2\frac{\d}{\d z_2}\right)S(\alpha_1z_2-\alpha_2z_1)=\frac{\ozeta(\alpha_1z_2-\alpha_2z_1)}{2}$, and therefore equation~\eqref{eq:equation with tA} is equivalent to
$$
S(\alpha_1 z_2-\alpha_2 z_1)\left(z_1\frac{\d}{\d z_1}+z_2\frac{\d}{\d z_2}\right)\frac{1}{T_{\alpha}(z_1+z_2)}=\frac{S(\alpha_1z_2-\alpha_2z_1)}{T_{\alpha}(z_1+z_2)}\sum_{h\ge 1}\tA_{0,1;\alpha,2h-1}(z_1+z_2)^{2h}.
$$
So we have to check the identity
$$
z\d_z\frac{1}{T_{\alpha}(z)}=\frac{1}{T_{\alpha}(z)}\sum_{h\ge 1}\tA_{0,1;\alpha,2h-1}z^{2h},
$$
which is true because
$$
\sum_{h\ge 1}\tA_{0,1;\alpha,2h-1}z^{2h}=\left.z\frac{\d}{\d z_1}\tA_{0,\alpha}(z_1,z)\right|_{z_1=0}=\left.T_\alpha(z)z\frac{\d}{\d z_1}\frac{1}{T_\alpha(z_1+z)}\right|_{z_1=0}=T_\alpha(z)z\d_z\frac{1}{T_\alpha(z)}.
$$
This concludes the proof of equation~\eqref{eq:A-coef}.

\medskip

Let us prove equation~\eqref{eq:B-coef}. We consider the formal power series~$\mcF^{\alpha,0}$, $\Omega^{\red;\alpha,0}$, $\mcF^{\DR;\alpha}$, and differential polynomials $\tOmega^{\alpha,0}$ defined in Section~\ref{subsection:relation between two hierarchies}. Since $T_\alpha(z)=T_\alpha(-z)$, the expansion of~$T_\alpha(z)$ has the form $T_{\alpha}(z)=\sum_{g\ge 0}T_{\alpha,g}z^{2g}$, $T_{\alpha,g}\in\mbQ$.

\begin{lemma}
We have $\tOmega^{\alpha,0}=-\sum_{g\ge 1}(\eps\mu)^{2g}T_{\alpha,g}w^\alpha_{2g-1}$.
\end{lemma}
\begin{proof}
We have to check that
\begin{gather}\label{eq:condition for reduced}
\Coef_{\eps^{2g}}\left.\frac{\d^n\left(T_\alpha(\eps\mu\d_x)\mcF^{\alpha,0}\right)}{\d t^{\alpha_1}_{d_1}\cdots\d t^{\alpha_n}_{d_n}}\right|_{t^*_*=0}=0\quad\begin{minipage}{10cm}
for any $g,n\ge 0$ and $\alpha_1,\ldots,\alpha_n\in\mbZ$, $d_1,\ldots,d_n\in\mbZ_{\ge 0}$\\
satisfying $\sum d_i\le 2g-1$.
\end{minipage}
\end{gather}
For $n\ge 2$, this immediately follows from degree reasons and the string equation
$$
\left(\frac{\d}{\d t^0_0}-\sum_{n\ge 0}t^\gamma_{n+1}\frac{\d}{\d t^\gamma_n}\right)\mcF^{\alpha,0}=t^\alpha_0.
$$
For $n=1$, the left-hand side of~\eqref{eq:condition for reduced} vanishes unless $\alpha_1=\alpha$, because $\sum_{m\ge 0}\gamma t^\gamma_m\frac{\d\mcF^{\P}}{\d t^\gamma_m}=0$. For $d\le 2g-2$, the vanishing 
$$
\Coef_{\eps^{2g}}\left.\frac{\d\left(T_\alpha(\eps\mu\d_x)\mcF^{\alpha,0}\right)}{\d t^{\alpha}_d}\right|_{t^*_*=0}=0
$$
follows again from degree reasons and the string equation for $\mcF^{\alpha,0}$.

\medskip

Then for $g\ge 1$, we compute
\begin{align*}
\Coef_{\eps^{2g}}\left.\frac{\d\left(T_\alpha(\eps\mu\d_x)\mcF^{\alpha,0}\right)}{\d t^{\alpha}_{2g-1}}\right|_{t^*_*=0}=&\mu^{2g}\sum_{j=0}^{g-1}T_{\alpha,j}\int_{\oM_{g-j,2}}\DR_{g-j}(\alpha,-\alpha)\psi_1^{2g-2j-1}+\mu^{2g}T_{\alpha,g}=\\
=&\mu^{2g}\sum_{j=0}^g T_{\alpha,j}\Coef_{z^{2g-2j}}\frac{S(\alpha z)}{S(z)}=\\
=&\Coef_{z^{2g}}\left(T_\alpha(z)\frac{S(\alpha z)}{S(z)}\right)=\\
=&0.
\end{align*}

\medskip

For $n=0$ and $g\ge 2$, the vanishing~\eqref{eq:condition for reduced} again follows from degree reasons and the string equation. It remains to check the case $n=0$ and $g=1$:
$$
\Coef_{\eps^{2}}\left.\left(T_\alpha(\eps\mu\d_x)\mcF^{\alpha,0}\right)\right|_{t^*_*=0}=\delta^{\alpha,0}\mu^2\bigg(\int_{\oM_{1,1}}\underbrace{\DR_1(0)}_{=-\lambda_1}+\underbrace{T_{0,1}}_{=\frac{1}{24}}\int_{\oM_{0,3}}1\bigg)=0.
$$
This completes the proof of the lemma.
\end{proof}

\medskip

The lemma implies that 
$$
w^{\red;\alpha}:=\frac{\d}{\d t^0_0}\Omega^{\red;\alpha,0}=\frac{\d}{\d t^0_0}\left(\mcF^{\alpha,0}+\sum_{g\ge 1}(\eps\mu)^{2g}T_{\alpha,g}w^{\top;\alpha}_{2g-1}\right)=T_\alpha(\eps\mu\d_x)w^{\top;\alpha},
$$
and from~\eqref{eq:general formula for Omega} and~\eqref{eq:A-coef} it then follows that 
\begin{gather}\label{eq:equation for wred}
\frac{w^{\red;\alpha}}{\d t^0_1}=\frac{1}{2}\sum_{\alpha_1+\alpha_2=\alpha}\d_x\left(e^{-\frac{\alpha_2}{2}\eps\mu\d_x}w^{\red;\alpha_1}\cdot e^{\frac{\alpha_1}{2}\eps\mu\d_x}w^{\red;\alpha_2}\right)+\sum_{g\ge 1}\eps^{2g}\mu^{2g-2}B_{\alpha,g}w^{\red;\alpha}_{2g+1}.
\end{gather}
Since $u^{\str;\alpha}$ satisfies the noncommutative KdV hierarchy (after the identification $t^0_n=t_n$), we also have
\begin{gather}\label{eq:equation for ustr}
\frac{u^{\str;\alpha}}{\d t^0_1}=\frac{1}{2}\sum_{\alpha_1+\alpha_2=\alpha}\d_x\left(e^{-\frac{\alpha_2}{2}\eps\mu\d_x}u^{\str;\alpha_1}\cdot e^{\frac{\alpha_1}{2}\eps\mu\d_x}u^{\str;\alpha_2}\right)+\frac{1}{12}\eps^2 u^{\str;\alpha}_{xxx}.
\end{gather}
Define an operator $G_\alpha\colon\mbC[[\eps,\mu,t^*_*]]\to\mbC[[\eps,\mu,z]]$ by
$$
G_\alpha(f):=\sum_{a\ge 0}\left.\frac{\d f}{\d t^\alpha_a}\right|_{t^*_*=0}z^a,\quad f\in\mbC[[\eps,\mu,t^*_*]].
$$
Let us apply the operator $G_\alpha$ to both sides equations~\eqref{eq:equation for wred} and~\eqref{eq:equation for ustr}.

\medskip

Let us recall the string and dilaton equations for $\Omega^{\red;\alpha,0}$ and $\mcF^{\DR;\alpha}$:
\begin{align*}
&\left(\frac{\d}{\d t^0_0}-\sum_{n\ge 0}t^\gamma_{n+1}\frac{\d}{\d t^\gamma_n}\right)\Omega^{\red;\alpha,0}=\left(\frac{\d}{\d t^0_0}-\sum_{n\ge 0}t^\gamma_{n+1}\frac{\d}{\d t^\gamma_n}\right)\mcF^{\DR;\alpha}=t^\alpha_0,\\
&\left(\frac{\d}{\d t^0_1}-\eps\frac{\d}{\d\eps}-\sum_{n\ge 0}t^\gamma_n\frac{\d }{\d t^\gamma_n}+1\right)\Omega^{\red;\alpha,0}=\left(\frac{\d}{\d t^0_1}-\eps\frac{\d}{\d\eps}-\sum_{n\ge 0}t^\gamma_n\frac{\d }{\d t^\gamma_n}+1\right)\mcF^{\DR;\alpha}=0.
\end{align*}
From~\eqref{eq:one-point equality} we then obtain that
\begin{gather*}
G_\alpha\left(\frac{w^{\red;\alpha}}{\d t^0_1}\right)=G_\alpha\left(\frac{u^{\str;\alpha}}{\d t^0_1}\right),\qquad G_\alpha(w^{\res;\alpha}_n)=z^n G_\alpha(w^{\res;\alpha})=z^n G_\alpha(u^{\str;\alpha})=G_\alpha(u^{\str;\alpha}_n).
\end{gather*}
Using also that $w^{\red;\alpha}_n|_{t^*_*=0}=u^{\str;\alpha}_n|_{t^*_*=0}=\delta^{\alpha,0}_{n,1}$, we get
\begin{multline*}
G_\alpha\left[\sum_{\alpha_1+\alpha_2=\alpha}\d_x\left(e^{-\frac{\alpha_2}{2}\eps\mu\d_x}w^{\red;\alpha_1}\cdot e^{\frac{\alpha_1}{2}\eps\mu\d_x}w^{\red;\alpha_2}\right)\right]=\\
=G_\alpha\left[\sum_{\alpha_1+\alpha_2=\alpha}\d_x\left(e^{-\frac{\alpha_2}{2}\eps\mu\d_x}u^{\str;\alpha_1}\cdot e^{\frac{\alpha_1}{2}\eps\mu\d_x}u^{\str;\alpha_2}\right)\right].
\end{multline*}
So we conclude that
$$
\underline{\left(\sum_{g\ge 1}\eps^{2g}\mu^{2g-2}B_{\alpha,g}z^{2g+1}-\frac{\eps^2}{12}z^3\right)}G_\alpha(u^{\str;\alpha})=0.
$$
By \cite[Theorem~4.9]{BS22}, we have 
$$
\Coef_{\eps^2\mu^0}\left.\frac{\d u^{\str;\alpha}}{\d t^\alpha_3}\right|_{t^*_*=0}=\Coef_{a^2}\int_{\oM_{1,2}}\lambda_1\DR_1(a,-a)=\frac{1}{24}.
$$
So $G_\alpha(u^{\str;\alpha})\ne 0$, and therefore the underlined expression above is equal to zero. This proves equation~\eqref{eq:B-coef}, completes the proof of Proposition~\ref{proposition:main}, and thus Theorem~\ref{theorem:main} is proved.

\end{document}